\documentclass[11pt, reqno]{amsart}

\usepackage{pstricks}
\usepackage{times}
\usepackage{comment}
\usepackage{xcolor}

\usepackage[numbers, sort&compress]{natbib}

\usepackage[breaklinks]{hyperref}

\usepackage{amsmath,amssymb,amsthm,mathtools}
\usepackage{hyperref}
\usepackage{enumitem}

\hypersetup{colorlinks=true, linkcolor=blue, citecolor=blue, urlcolor=blue}

\usepackage[utf8]{inputenc}
\usepackage{amsmath}
\usepackage{amsfonts}
\usepackage{bbm}
\usepackage{mathtools}
\usepackage{tikz}
\usetikzlibrary{decorations.pathreplacing}
\usepackage{amsthm}
\usepackage[english]{babel}
\usepackage{fancyhdr}
\usepackage{amssymb}
\usepackage{enumitem}
\usepackage{qtree}
\usepackage{mathrsfs}

\numberwithin{equation}{section}

\newtheorem{definition}{Definition}
\newtheorem{theorem}{Theorem}
\newtheorem{proposition}{Proposition}
\newtheorem{lemma}{Lemma}

\theoremstyle{remark}
\newtheorem{remark}{Remark}

\newcommand{\R}{\mathbb{R}}

\newcommand{\E}{\mathbb{E}}
\newcommand{\Prob}{\mathbb{P}}
\newcommand{\1}{\mathbf{1}}

\newcommand{\SO}{\mathrm{SO}}
\newcommand{\so}{\mathfrak{so}}
\newcommand{\Id}{\mathrm{Id}}
\newcommand{\HS}{\mathrm{HS}}
\newcommand{\op}{\mathrm{op}}
\newcommand{\Law}{\mathcal{L}}
\newcommand{\TV}{\mathrm{TV}}
\newcommand{\KL}{\mathrm{KL}}
\newcommand{\Tr}{\mathrm{Tr}}
\newcommand{\Ad}{\mathrm{Ad}}
\newcommand{\diag}{\mathrm{diag}}
\newcommand{\mix}{\mathrm{mix}}
\newcommand{\dd}{\mathrm{d}}
\newcommand{\eps}{\varepsilon}
\newcommand{\ip}[2]{\langle #1,#2\rangle}
\newcommand{\norm}[1]{\left\lVert #1\right\rVert}
\newcommand{\abs}[1]{\left\lvert #1\right\rvert}

\begin{document}
\title[Kac's walk on rotation matrices mixes in $n^2 \log n$ steps]{Kac's walk on rotation matrices mixes in $n^2 \log n$ steps}
\author{Natesh S. Pillai}

\address{Department of Statistics, Harvard University}
\email{pillai@fas.harvard.edu}

\author{Aaron Smith}
\address{Department of Mathematics and Statistics, University of Ottawa}
\email{asmi28@uOttawa.ca}


\begin{abstract}
Kac’s walk on $\SO(n)$, introduced in \cite{hastings70}, is an important high-dimensional Markov chain with applications in statistical physics, statistics, cryptography, and computational sciences. Despite its simple transition rules, determining its total-variation mixing time has remained a  challenging problem for decades. A key obstacle is that the walk is not conjugation-invariant, placing it beyond the reach of classical Fourier-analytic techniques that apply to many related random walks on compact groups.

We prove that Kac’s walk mixes in total variation in $O(n^2 \log n)$ steps, matching the conjectured mixing time up to constants. The proof is based on a refined two-stage coupling. Building on our earlier work \cite{PillaiSmith}, the first stage contracts two copies of the chain to a small neighborhood via a Wasserstein coupling defined and analyzed in \cite{Oliveira}. 

Our main contribution is a new framework for analyzing the second-stage coupling. It can be viewed as a discrete analogue of Malliavin calculus for Markov chains. We represent the law of the chain as the pushforward of high-dimensional noise and prove quantitative non-degeneracy of the associated linearization using matrix martingale methods. This yields an approximately Gaussian distribution in the Lie algebra with well-conditioned covariance, allowing small group translations to be absorbed at negligible cost in total variation. Our approach provides a new general framework for studying mixing in high dimensional Markov chains in continuous state space with singular transition kernels.

\end{abstract}
\maketitle

\section{Introduction}

Introduced by Kac in the 1950s, Kac's random walk on the sphere is one of the canonical stochastic models for the Boltzmann equation \cite{kac1956foundations,MischlerMouhot2013}. The analogous walk on the special orthogonal group $\SO(n)$ was defined in the seminal paper introducing Markov chain Monte Carlo to statistics, where it was used as the most-developed sample algorithm \cite{hastings70}. Although the modern theory of mixing times had not been developed when  \cite{hastings70} was written, the last sentence of that paper asks essentially the same question that we study in this paper: how dependent are widely-spaced samples from Kac's walk, and how does this relate to computational efficiency? 

Since 1970, Kac-type dynamics have become a natural test case for a basic question that appears across probability, statistics, and computation: how rapidly can a high-dimensional constrained system be randomized using only local low-dimensional updates? 
The simplicity of the dynamics is an asset when using the walk as an algorithm: it makes both single-step computations (as in \cite{hastings70}) and downstream applications (as in \cite{sotirakitrap,JainPillaiSahSawhneySmith2022,VaikuntanathanZamir2025}) faster. The same simplicity, however, makes the walk harder to analyze than superficially similar walks on the same group \cite{HoughJiang,Porod,Rosenthal}. Kac-like dynamics have continued to be used in Bayesian statistics \cite{bornn2019moment}, and are important in developing Gibbs samplers for distributions with hard constraints \cite{Oliveira,PillaiSmith} and especially orthogonality constraints \cite{JauchHoffDunson2020}. More recently, the same dynamics have been used as steps in other algorithms, including fast dimension reduction \cite{AilonChazelle06,JainPillaiSahSawhneySmith2022}, trapdoored pseudorandom matrices in cryptography \cite{sotirakitrap}, parallel Kac-type constructions in quantum pseudorandomness \cite{LuQinSongYaoZhao2024,LuQinSongYaoZhao2025}, other applications in machine learning \cite{VaikuntanathanZamir2025}, and as a tool to prove stronger results in other areas of probability \cite{NgWalters15}.

Let $\SO(n)$ denote the group of $n\times n$ rotation matrices and let $N=\binom{n}{2}$ be the dimension of $\SO(n)$. Fix an ordering of the $N$ coordinate pairs
\[
[n]_{\neq}=\{(a,b):1\le a<b\le n\}.
\]
For $i\in\{1,\dots,N\}$ and $\theta\in[0,2\pi)$, let $R(i,\theta)$ denote the rotation by angle $\theta$ in the coordinate plane associated with the $i^\mathrm{th}$ element of $[n]_{\neq}$. The chain $\{X_t\}_{t\ge0}$ evolves by
\begin{equation}
\label{eq:kac-update}
X_{t+1}=R(i_t,\theta_t)X_t,
\end{equation}
where $i_t$ is uniform on $\{1,\dots,N\}$ and $\theta_t$ is uniform on $[0,2\pi)$. Thus each step acts in only one two-dimensional coordinate plane, while the stationary distribution is Haar measure on a manifold of dimension $N$. The closely-related walk on the sphere is obtained by taking the first column of $X_{t}$; this is also a Markov chain.

The total variation mixing time for Kac's walk on $\SO(n)$ has been a long standing open problem. The best bound to date,
\[
\tau_{\mix}=O(n^4\log n),
\]
was obtained in \cite[Theorem~1.1]{PillaiSmith}. Combined with the $\Omega(n^2)$ lower bound from \cite{PillaiSmith}, this left a factor $n^2$ gap in the main quantitative question for the chain. In this paper, we close this gap and prove an upper bound that is tight up to a factor of $O(\log(n))$:

\begin{theorem}
\label{thm:main}
The total-variation mixing time of Kac's walk on $\SO(n)$ satisfies
\[
\tau_{\mix}=O(n^2\log n).
\]
\end{theorem}

A nearly-matching lower bound of $\Omega(n^{2})$ can be obtained by dimension-counting: $\SO(n)$ has dimension $N$, while each step is constrained to a union of 1-dimensional submanifolds. We conjecture that our upper bound of $O(n^{2} \log(n))$ is tight.

\subsection{Our contribution and its connection to Malliavin calculus}

There is a large literature on bounding the mixing times of Markov chains in many contexts. The theory is best developed for chains on discrete spaces \cite{levin2017markov}, while in continuous state spaces the best available results are, as noted in \cite{Lacoin2021ParticleSystems}, “very far from being as nice as in the finite case”; see also the discussions in \cite{CaputoLabbeLacoin2020AdjacentSimplex,LabbePetit2025BiasedAdjacentSimplex,CaputoLabbeLacoin2022GibbsNablaPhi}.

A key difficulty in continuous state spaces is the lack of a general mechanism for upgrading convergence in weaker metrics to total variation. In contrast, in discrete settings this upgrade typically follows from a spectral gap bound at logarithmic cost in the inverse of the smallest stationary probability. For Markov chains like Kac's walk that make low-dimensional updates, it is well-known that knowledge of the highest eigenvalue ``does not improve the mixing time" in general \cite{PakSidenko}. Previous work on these upgrades has either taken advantage of symmetries that apply to very narrow families of Markov chains \cite{Rosenthal,Porod,HoughJiang} or has effectively sought a ``warm start" (introduced in \cite{LovaszSimonovits1993RandomWalksConvexBody}) via truncation or perturbation estimates \cite{JiangSO,PillaiSmith}.

The main difficulty in proving total-variation mixing for Kac's walk is not the absence of
contraction.  Wasserstein contraction is already known at the conjectural time scale by
Oliveira's coupling \cite{Oliveira}.  The difficulty is that the transition kernel is singular.
At each step the chain moves only along a one-dimensional family of rotations, while the
state space $\SO(n)$ has dimension $N=\binom n2$.  Thus, for short and moderate times,
the usual smoothing mechanisms available for Markov chains with everywhere positive
transition densities are absent.  One must instead understand how many singular, highly
structured updates combine to create a distribution that is spread in all tangent directions,
and one must do so quantitatively enough to compare nearby translates in total variation.

This problem is closely analogous to the classical problem of hypoellipticity.
Consider the Stratonovich stochastic differential equation
\begin{equation}
\label{eq:classical-hypoelliptic-sde}
    \dd X_t
    =
    V_0(X_t)\,\dd t
    +
    \sum_{\alpha=1}^r V_\alpha(X_t)\circ \dd W_t^\alpha ,
\end{equation}
on a smooth manifold \(M\), where \(V_0,V_1,\ldots,V_r\) are smooth vector
fields and \(W^\alpha\) are independent Brownian motions. Its generator is
\[
    \mathcal L
    =
    V_0+\frac12\sum_{\alpha=1}^r V_\alpha^2.
\]
Even when the noise vector fields $V_1,\ldots,V_r$ fail to span the tangent space pointwise, their iterated Lie brackets can generate the missing directions. Hörmander’s theorem states that if, for every $x \in M$, the Lie algebra generated by $V_1,\ldots,V_r$ together with their commutators with the drift $V_0$ spans the full tangent space $T_xM$, then the associated diffusion has smooth transition densities; equivalently, $\partial_t - \mathcal{L}$ is hypoelliptic \cite{Hormander1967,Hairer2011Hormander}.

Malliavin calculus gives a probabilistic route to obtaining this regularity
\cite{Malliavin1978,Nualart2006,Hairer2011Hormander,KusuokaStroock1984,KusuokaStroock1985,Norris1986}.
Differentiation of the solution map $X_t$ with respect to the driving noise yields a
linearized flow. If \(J_{s,t}:T_{X_s}M\to T_{X_t}M\) denotes this flow, then, in
local coordinates,
\[
    \dd J_{s,t}
    =
    DV_0(X_t)J_{s,t}\,\dd t
    +
    \sum_{\alpha=1}^r DV_\alpha(X_t)J_{s,t}\circ \dd W_t^\alpha,
    \qquad
    J_{s,s}=\Id .
\]
The Malliavin derivative of the endpoint in the \(\alpha\)-th noise direction at
time \(s\) is
\[
    D_s^\alpha X_t
    =
    J_{s,t}V_\alpha(X_s),
    \qquad 0\le s\le t.
\]
The associated Malliavin covariance matrix is
\begin{equation}
\label{eq:classical-malliavin-covariance}
    \mathcal M_t
    =
    \sum_{\alpha=1}^r
    \int_0^t
    \bigl(J_{s,t}V_\alpha(X_s)\bigr)
    \bigl(J_{s,t}V_\alpha(X_s)\bigr)^\top
    \,\dd s .
\end{equation}

Non-degeneracy of \(\mathcal M_t\) means that the injected Brownian noise has
reached every tangent direction by time \(t\). Under Hörmander's condition, it turns 
out that \(\mathcal M_t\) is invertible, and Malliavin integration by parts
then converts this covariance information into regularity of the law; see, for
example, Equation~3.4 of \cite{Hairer2011Hormander}. The classical scheme is
therefore to differentiate the endpoint with respect to the noise, establish
non-degeneracy of the resulting covariance, and then use integration by parts to
deduce regularity.

Our argument follows a similar overall pattern, but in a discrete, finite-dimensional setting tailored to Markov chains with singular transition kernels. In place of a Brownian path and an infinitesimal generator satisfying the celebrated Hörmander bracket condition, the randomness enters through the sequence of coordinate planes and rotation angles. We introduce a small Gaussian perturbation of the angles in a fresh block of updates, which plays the role of the driving noise in this setting.

More precisely, after a first-stage Wasserstein coupling due to Oliveira
\cite{Oliveira}, two copies of the walk can be brought within a very small
Hilbert--Schmidt distance after \(O(N\log n)\) steps, with high probability.
This produces closeness, but not coalescence.  To upgrade this to total-variation
overlap, we expose a fresh block
\[
    A=((i_1,\theta_1),\ldots,(i_m,\theta_m)),
    \qquad
    m=\lceil C_1N\log N\rceil,
\]
and perturb all angles in this block by an independent Gaussian vector
\[
    \delta=(\delta_1,\ldots,\delta_m)\sim \mathcal N(0,\sigma_n^2 I_m).
\]
Adding such a perturbation preserves the marginal law of Kac's walk, because the
base angles are uniform modulo \(2\pi\).  Thus the perturbation gives us a source
of smoothing without changing the Markov chain.

The endpoint perturbation is encoded by
\[
    H_A(\delta)=L_A(\delta)L_A(0)^{-1},
    \qquad
    u_A(\delta)=\log H_A(\delta)\in\so(n),
\]
where \(L_A(\delta)\) is the product of the perturbed rotations.  The map
\[
    u_A:\R^m\to\so(n)
\]
is the finite-dimensional analogue of the endpoint map in Malliavin calculus.
Its derivative at the origin,
\[
    J_A=\dd u_A(0),
\]
plays the role of the Malliavin derivative, and
\[
    M_A=J_AJ_A^\top
\]
plays the role of the Malliavin covariance matrix.

The key estimate in the paper is a quantitative non-degeneracy bound for \(M_A\).
The columns of \(J_A\) are transported coordinate-plane directions.  Thus \(M_A\)
is a sum of rank-one projections onto these transported directions, directly
mirroring the classical covariance formula
\eqref{eq:classical-malliavin-covariance}.  The important new observation is that,
after conditioning on the future of the block and reversing time, these rank-one
terms have a matrix martingale difference structure.  Tropp's matrix Freedman
inequality then gives, with high probability,
\[
    \frac{m}{N}I_N
    \preceq
    M_A
    \preceq
    3\frac{m}{N}I_N.
\]
This is the discrete analogue of non-degeneracy of the Malliavin covariance:
although each update is one-dimensional, the transported update directions spread
across all \(N=\dim\SO(n)\) tangent directions over a block of length
\(O(N\log N)\).

Once this covariance estimate is available, the remainder of the second stage is
perturbative.  On the scale of the Gaussian perturbation, the logarithmic map has
the expansion
\[
    u_A(\delta)=J_A\delta+r_A(\delta),
\]
with a quadratic remainder.  Hence the law of \(u_A(\delta)\) is close in total
variation to a non-degenerate Gaussian measure on \(\so(n)\) with covariance
\(\sigma_n^2M_A\).  Small right translations of the group then become, in
logarithmic coordinates, small shifts of this Gaussian together with a controlled
Baker--Campbell--Hausdorff error.  Since the first-stage coupling has already
made the two chains \(N^{-4}\)-close, the resulting translation cost is \(o(1)\).

This perspective also explains the improvement over our previous work
\cite{PillaiSmith}.  There, the second stage perturbed only \(N\) carefully
chosen updates and required a delicate lower bound on the smallest singular value
of a square random Jacobian.  Here we perturb all updates in a shorter block of
length \(O(N\log N)\).  The resulting Jacobian is rectangular, but its covariance
has a clean martingale structure, making the random-matrix analysis substantially
more tractable.

Although several estimates in the present paper use special features of Kac's
walk, the general mechanism is more flexible; see Remark \ref{rem:keymallrm}.
Therefore we expect this point of view
to be useful for other high-dimensional Markov chains on continuous state spaces
with singular transition kernels.

Finally, our proof is conceptually related to the mechanism behind ergodicity for
stochastic systems with degenerate forcing, as in \cite{HairerMattingly2006}.  In
both settings, randomness is injected only in a restricted set of directions, and
one studies how the dynamics transports infinitesimal perturbations of the random
input.  In \cite{HairerMattingly2006}, this leads to asymptotic smoothing despite
a genuinely degenerate Malliavin matrix.  In our finite-dimensional setting, the
same philosophy yields a non-degenerate Gaussian approximation in logarithmic
coordinates.
\subsection{Related Work}
Much of the early work on mixing times of Markov chains focused on random walks on groups \cite{DiaconisShahshahani1981,Aldous1983}. For conjugation-invariant random walks on compact groups, powerful Fourier-analytic methods are available, and in several cases they lead to sharp total-variation estimates and cutoff phenomena \cite{Diaconiscut}, including for random walks on groups that are similar to Kac's walk \cite{Rosenthal,Porod,HoughJiang}. Kac's walk on $\SO(n)$ is different: its transition rule is tied to the ambient coordinate axes and is not conjugation invariant, so the classical representation-theoretic approach does not apply directly. 

The first line of work on Kac's walks studied their spectral gaps. For the walk on the sphere, Janvresse identified the spectral gap up to constants \cite{Jan01}. For the walk on $\SO(n)$, Diaconis and Saloff-Coste gave the first bounds that were polynomial in $n$ using a comparison argument, and Janvresse showed that the spectral gap is of order $n^{-2}$ \cite{DSC00,Jan03}. The exact spectral gap for both the sphere walk and the $\SO(n)$ walk was then identified by \cite{CCL03}, and \cite{Mas03} computed the full spectrum; \cite{Cap08} later gave a streamlined and more general spectral-gap argument for Kac-type collision processes.  Even with exact computations of the spectral gap and even the full spectrum, the lack of conjugacy-invariance meant that it was difficult to obtain bounds on the mixing time in total variation. As illustration, the first paper to obtain strong estimates of the spectral gap also gave the first bounds on the mixing time of the walk, but was unable to get a bound stronger than $e^{O(n^{2})}$ \cite{DSC00}. Despite the full characterization of the spectrum in the following decade, this estimate was not improved to a polynomial mixing bound until \cite{JiangSO}.

In parallel to work on estimating the spectral gaps of Kac's walks, there was work on estimating their mixing time in stronger metrics. On the sphere, the first polynomial upper bound on the mixing time was computed in \cite{Jiang2012KacSphere}, and the correct mixing time was computed by the present authors in \cite{PillaiSmith2017KacSphere}. On $\SO(n)$, The first strong result was due to Pak and Sidenko, who proved an $O(n^{5/2}\log n)$ weak-convergence bound \cite{PakSidenko}. Oliveira later showed that the walk mixes in Wasserstein distance in $O(n^2\log n)$ steps, within a logarithmic factor of optimal \cite{Oliveira}. The first polynomial upper bound on the mixing time was computed in \cite{JiangSO}, and the present authors later improved this to $O(n^4\log n)$ \cite{PillaiSmith}. The present note improves the estimate to $O(n^2\log n)$, which is conjectured to be the correct answer.

There is also substantial work on closely related random walks on compact groups. Rosenthal analyzed random fixed-angle rotations on $\SO(N)$ using character theory, Porod studied the random-reflection walk, and Hough and Jiang proved cutoff for the uniform-plane Kac walk, a conjugation-invariant variant in which the rotation plane is sampled uniformly from the Grassmannian \cite{Rosenthal,Porod,HoughJiang}. 

\section{Notation and Geometry of $\SO(n)$}
\label{sec:notation}

\subsection{Basic Notation}
Let $K$ denote the one-step transition kernel of Kac's walk and let $K^{t}$ denote the $t$-step transition kernel. Let $\mu$ denote normalized Haar measure on $\SO(n)$, which is the stationary distribution of Kac's walk. In this notation, and in most of the paper, we suppress dependence on $n$ unless it is necessary.

We  write $\Law(Z)$ for the law of a random variable $Z$. For a distribution $\mu$, we write $Z \sim \mu$ as shorthand for $\Law(Z) = \mu$. For probability measures $\nu_1,\nu_2$ on the same measurable space, we write
\[
\norm{\nu_1-\nu_2}_{\TV}=\sup_A \abs{\nu_1(A)-\nu_2(A)}.
\]
The total-variation mixing time of Kac's walk is
\[
\tau_{\mix}=\inf\Bigl\{t\ge 0: \sup_{x\in \SO(n)} \norm{K^t(x,\cdot)-\mu}_{\TV}\le \frac14\Bigr\}.
\]
For $d$ a metric on $\SO(n)$, set
\[
W_{d,2}(\nu_1,\nu_2)=\inf\Bigl\{\bigl(\E[d(X,Y)^2]\bigr)^{1/2}: \Law(X)=\nu_1,\ \Law(Y)=\nu_2\Bigr\},
\]
where the infimum is over all couplings $(X,Y)$ of $\nu_1$ and $\nu_2$. Define
\[
\tau_{d,2}(\eps)=\inf\Bigl\{t\ge 0: \sup_{x,y\in \SO(n)} W_{d,2}\bigl(K^t(x,\cdot),K^t(y,\cdot)\bigr)\le \eps\Bigr\}.
\]
For a measurable map $T$ and a probability measure $\nu$, $T_\#\nu$ denotes the push-forward measure, defined by
\[
(T_\#\nu)(B)=\nu\bigl(T^{-1}(B)\bigr)
\]
for measurable sets $B$. If $p\in\R^d$ and $q$ is a vector in a Euclidean space $V$, then $p\otimes q:\R^d\to V$ denotes the rank-one map $u\mapsto \langle p,u\rangle q$.

Throughout the paper, we use the polynomial scales
\begin{equation}
\label{eq:scales}
\omega_n = N^{-4},
\qquad
\sigma_n = N^{-3},
\qquad
m = \lceil C_1 N\log N\rceil
\end{equation}
where $C_1>10$ is a fixed constant\footnote{We have not tried to optimize the value of $C_1$} independent of $N$.

\subsection{Geometry}

Let $e_1,\dots,e_m$ denote the standard basis vectors of $\R^m$. We write $M_n(\R)$ for the space of real $n\times n$ matrices, and for $M\in M_n(\R)$ we write $\norm{M}_{\op}$ for its operator norm with respect to the Euclidean norm on $\R^n$. If $F$ is a $C^2$ map between finite-dimensional Euclidean spaces, then $\dd F(x)$ and $\dd^2F(x)$ denote the first and second Fr\'echet derivatives at $x$; thus $\dd F(x)[u]$ is linear in $u$ and $\dd^2F(x)[u,v]$ is bilinear in $(u,v)$. We write $\norm{\dd F(x)}_{\op}$ and $\norm{\dd^2F(x)}_{\op}$ for the induced operator norms, using the Hilbert-Schmidt norm when the codomain is $\so(n)$ or $M_n(\R)$, and the Euclidean norm otherwise. For symmetric matrices
$A,B$, we write $A\preceq B$ if $B-A$ is positive semidefinite, and we write
$\lambda_{\min}(A)$ and $\lambda_{\max}(A)$ for the smallest and largest eigenvalues of $A$. 

For $1\le k<\ell\le n$, let $E_{k\ell}$ be the matrix with a $1$ in position $(k,\ell)$ and $0$ elsewhere, and define
\[
a_{k\ell} = \frac{1}{\sqrt{2}}\bigl(E_{k\ell}-E_{\ell k}\bigr).
\]
Using our fixed ordering of $[n]_{\neq}$, we write these matrices in order as
\[
a_1,\dots,a_N.
\]
They form a basis of the tangent space $\so(n)$ of $SO(n)$ at the identity. Under the Hilbert--Schmidt inner product
\[
\ip{u}{v}_{\HS} = \Tr(u^\top v),
\]
this is an orthonormal basis. Furthermore, since $\SO(n)$ is a matrix Lie group, its Lie-group exponential agrees with the usual matrix exponential, and for the coordinate-plane rotation associated with $a_i$ we have
\[
R(i,\theta) = \exp \!\bigl(\sqrt{2}\,\theta a_i\bigr).
\]

For $g\in \SO(n)$ and $u\in \so(n)$, denote by $\Ad(g)u=gug^{-1}$ the adjoint action of $\SO(n)$ on $\so(n)$. Write $D_{\HS}$ for the bi-invariant Riemannian metric on $\SO(n)$ induced by
$\ip{\cdot}{\cdot}_{\HS}$.
Since the ambient Hilbert--Schmidt norm of a differentiable curve is bounded by its speed,
\begin{equation}
\label{eq:ambient-vs-riem}
\norm{x-y}_{\HS} \le D_{\HS}(x,y)
\end{equation}
for all $x,y\in \SO(n)$. We will do many calculations on $\SO(n)$ by locally linearizing small perturbations and then doing easier calculations on the tangent space. To justify this, we note the following result which is an immediate corollary of Theorem 1.31 of \cite{high:FM}:

\begin{lemma}[Local logarithm near the identity]
\label{lem:local-log}
There exist universal constants $\rho_0,c_{\log}>0$ such that, for all $n$ and all $g\in \SO(n)$ satisfying
$D_{\HS}(\Id,g)\le \rho_0$, the principal matrix logarithm $h=\log g$ is well defined and
\[
\norm{h}_{\HS}\le c_{\log}D_{\HS}(\Id,g).
\]
\end{lemma}

The principal matrix logarithm $h$ of $g \in SO(n)$, when it exists, satisfies $g = \exp(h)$. Lemma \ref{lem:local-log} tells us that we can take logs (and thus ``solve" the equation $g = \exp(h)$) as necessary. From this point onward, whenever we write $\log$ of a group element near the identity, we mean the principal matrix logarithm.

\section{Two-stage coupling}
\subsection{The first-stage coupling}

The first stage of our two-stage coupling will be based on the following proposition, which is essentially an immediate corollary of the main result of \cite{Oliveira}:

\begin{proposition}[Contractive coupling]
\label{prop:scaffold}
Fix $B>0$.
There exists $C_0(B)<\infty$ such that for every $x,y\in \SO(n)$ one can couple two copies
$(X_t,Y_t)$ of Kac's walk with $X_0=x$ and $Y_0=y$ so that, for all
\begin{equation}
\label{eq:t0-def}
t_0 \geq C_0(B)\,N\log n,
\end{equation}
we have
\begin{equation}
\label{eq:scaffold-bound}
\Prob\!\bigl(D_{\HS}(X_{t_0},Y_{t_0})>N^{-B}\bigr)\le N^{-2}.
\end{equation}
\end{proposition}

\begin{proof}
Theorem 1 of \cite{Oliveira} states that, for all $\eps > 0$, all $t\ge \bigl\lceil n^2\log(\pi\sqrt{n}/\eps)\bigr\rceil$ and all $z \in \SO(n)$,
\[
W_{D_{\HS},2}\bigl(K^t(z,\cdot),\mu\bigr)\le \eps.
\]

Applying this with $\eps=\tfrac12 N^{-B-1}$ shows that there exists $C_0(B)$ such that for $t_0$ as in
\eqref{eq:t0-def} and all $z \in SO(n)$
\[
W_{D_{\HS},2}\bigl(K^{t_0}(z,\cdot),\mu\bigr)\le \frac12 N^{-B-1}.
\]
Therefore, for all $x,y\in \SO(n)$,
\begin{align*}
W_{D_{\HS},2}\bigl(K^{t_0}(x,\cdot),K^{t_0}(y,\cdot)\bigr) 
\le
W_{D_{\HS},2}\bigl( &K^{t_0}(x,\cdot),\mu\bigr) \\
&+
W_{D_{\HS},2}\bigl(\mu,K^{t_0}(y,\cdot)\bigr) 
\le N^{-B-1}.
\end{align*}
By the definition of Wasserstein distance, there exists a coupling of $(X_{t_0},Y_{t_0})$ with
\[
\E\bigl[D_{\HS}(X_{t_0},Y_{t_0})^2\bigr]\le N^{-2B-2}.
\]
Markov's inequality then gives
\[
\Prob\!\bigl(D_{\HS}(X_{t_0},Y_{t_0})>N^{-B}\bigr)
\le
N^{2B}\E\bigl[D_{\HS}(X_{t_0},Y_{t_0})^2\bigr]
\le N^{-2},
\]
which is \eqref{eq:scaffold-bound} and we are done.
\end{proof}

\subsection{Second stage coupling}
\label{subsec:global-notation}
We now define our second stage of the coupling that enables the Wasserstein closeness from this first
stage to total variation.
The key idea is to represent the endpoint of the chain as a \emph{perturbation of a fixed block of updates}, and to study this perturbation in logarithmic coordinates.

\begin{definition} (Perturbation of Kac's walk) \label{defn:randupdate}
Fix $m = \lceil C_1 N\log N\rceil$ as in \eqref{eq:scales}, and let
\begin{equation}
\label{eq:A-def}
A = \bigl((i_1,\theta_1),\dots,(i_m,\theta_m)\bigr)
\end{equation}
be a sequence of i.i.d.\ updates, where $i_t$ is uniform on $\{1,\dots,N\}$ and $\theta_t$ is uniform on $[0,2\pi)$. For $\delta\in\R^m$, define the perturbed product
\begin{equation}
\label{eq:L-def}
L_A(\delta)
=
R(i_m,\theta_m+\delta_m)\cdots R(i_1,\theta_1+\delta_1).
\end{equation}
\end{definition}

\begin{lemma}
\label{lem:marginal}
Let $\Delta=(\Delta_1,\dots,\Delta_m)$ be any random vector independent of the base angles
$(\theta_t)_{t=1}^m$.
Then for every $z\in \SO(n)$,
\[
L_A(\Delta)z \sim K^m(z,\cdot).
\]
\end{lemma}

\begin{proof}
Condition on $(i_t)_{t=1}^m$, on $\Delta$, and on $z$.
Since addition mod $2\pi$ preserves the uniform law on $[0,2\pi)$,
each $\theta_t+\Delta_t$ is uniform on $[0,2\pi)$.
Conditional on $\Delta$, the variables $\theta_t+\Delta_t$ remain independent because the
base angles $\theta_t$ are independent and are independent of $\Delta$.
Therefore
\[
\bigl(i_t,\theta_t+\Delta_t\bigr)_{t=1}^m
\]
has exactly the same law as the usual update sequence for $m$ steps of Kac's walk,
so $L_A(\Delta)z$ has law $K^m(z,\cdot)$.
\end{proof}

By Lemma~\ref{lem:marginal},  we may endow the vector $\Delta$ with \emph{any} distribution as long as it is independent of $A$, while still preserving the distribution of the $m$-step update to Kac's walk. Inspired by the role of Brownian motion in Malliavin calculus, we choose this distribution to be Gaussian:
\begin{definition} (Gaussian perturbation of the driving noise.) \label{defn:randupdate2}
For the random perturbation, let \footnote{As is conventional, we write $\Delta$ for 
a random variable and use $\delta$ to denote its realization.}
\begin{equation}\label{eq:Deltadef}
\Delta\sim \mathcal N(0,\sigma_n^2 I_m)
\end{equation}
independent of $A$ in \eqref{eq:A-def} and $\sigma_n = N^{-3}$ as in \eqref{eq:scales}.
\end{definition}
Definitions \ref{defn:randupdate} and \ref{defn:randupdate2} complete our second stage coupling.

\section{Proof sketch of the second stage coupling}
\label{SecProofSketch}

We now summarize the mechanism that upgrades the Wasserstein contraction from
Proposition~\ref{prop:scaffold} to total-variation overlap. Starting from arbitrary $x,y\in\SO(n)$, Proposition~\ref{prop:scaffold} gives a
coupling such that after
\[
t_0=O(N\log n)
\]
steps,
\[
D_{\HS}(X_{t_0},Y_{t_0})\le \omega_n
\]
with high probability. On this event, let
\[
g=X_{t_0}^{-1}Y_{t_0}.
\]
By bi-invariance of $D_{\HS}$ and Lemma~\ref{lem:local-log}, for all large $n$ we
may write
\[
g=\exp(h),
\qquad
\|h\|_{\HS}\le c_{\log}\omega_n.
\]
Thus the second stage reduces to showing that a fresh block of updates can absorb
right multiplication by such a small element.

For the second stage, we expose an independent block
\[
A=((i_1,\theta_1),\dots,(i_m,\theta_m)),
\qquad
m=\lceil C_1N\log N\rceil,
\]
and as described in Definition \ref{defn:randupdate2}, perturb all angles in the block by an independent Gaussian vector
\[
\Delta\sim\mathcal N(0,\sigma_n^2I_m).
\]
Write
\[
H_A(\delta)=L_A(\delta)L_A(0)^{-1},
\qquad
\nu_A=\Law(H_A(\Delta)\mid A).
\]
If
\[
z=L_A(0)X_{t_0},
\]
then Lemma~\ref{lem:marginal} shows that the two second-stage conditional laws are
\[
\Law(L_A(\Delta)X_{t_0}\mid A)=(T_z)_\#\nu_A,
\qquad
\Law(L_A(\Delta)Y_{t_0}\mid A)=(T_{zg})_\#\nu_A.
\]
Since total variation is invariant under right translation and the adjoint action
preserves the Hilbert--Schmidt norm, it is enough to prove that
\[
\|\nu_A-(T_{\exp(h)})_\#\nu_A\|_{\TV}=o(1)
\]
uniformly for $\|h\|_{\HS}\le c_{\log}\omega_n$.

To study \(\nu_A\), we pass to logarithmic coordinates. On the high-probability
event
\[
E=\{\|\Delta\|_2\le 2\sigma_n\sqrt m\},
\]
the logarithm
\[
U_A=u_A(\Delta)=\log H_A(\Delta)
\]
is well-defined. Proposition~\ref{prop:quadratic} gives the local expansion
\[
u_A(\delta)=J_A\delta+r_A(\delta),
\qquad
J_A=\dd u_A(0),
\]
with a quadratic remainder. Thus \(U_A\) is well approximated by the linear term
\(J_A\Delta\). The covariance of this linear approximation is
\[
M_A=J_AJ_A^\top.
\]
Proposition~\ref{prop:covariance} shows that with high probability over the block,
\[
\frac{m}{N}I_N\preceq M_A\preceq 3\frac{m}{N}I_N.
\]
Hence the Gaussian measure
\[
\gamma_A=\mathcal N(0,\Sigma_A),
\qquad
\Sigma_A=\sigma_n^2M_A,
\]
is non-degenerate in every tangent direction. The matrix $M_A$ thus plays the role of the Malliavin covariance matrix. Proposition~\ref{prop:gaussian-approx}
then shows that
\[
\|\Law(U_A\mid A,E)-\gamma_A\|_{\TV}=o(1).
\]

The final step is to transfer this Gaussian approximation on $\so(n)$ back to the group.
Proposition~\ref{prop:gaussian-bch} shows that if a law on \(\SO(n)\) is the
exponential pushforward of a measure on \(\so(n)\) that is close to a centered
Gaussian, then right multiplication by \(\exp(h)\) changes the group law by the
Gaussian shift cost
\[
\|\Sigma_A^{-1/2}h\|_{\HS}
\]
plus a Baker--Campbell--Hausdorff error of order
\[
\sqrt N\,\|h\|_{\HS},
\]
up to exponentially small tails. Since \(\Sigma_A\simeq \sigma_n^2(m/N)I_N\) and \(\|h\|_{\HS}\le c_{\log}\omega_n\), both terms are \(o(1)\).
Combining the above, Proposition~\ref{prop:block-translate} then yields
\[
\|\nu_A-(T_{\exp(h)})_\#\nu_A\|_{\TV}=o(1).
\]

The rest of the paper proves these ingredients in order:
Section~\ref{sec:perturbation} proves the quadratic expansion
(Proposition~\ref{prop:quadratic}),
Section~\ref{sec:covariance} proves the covariance estimate
(Proposition~\ref{prop:covariance}),
Section~\ref{sec:transport} proves the Gaussian approximation
(Proposition~\ref{prop:gaussian-approx}),
Section~\ref{sec:group-transport} proves the group-transport estimates
(Propositions~\ref{prop:gaussian-bch} and \ref{prop:block-translate}),
and Section~\ref{sec:proof-main} assembles the two-stage coupling to prove
Theorem~\ref{thm:main}.

\section{The perturbation block and its linearization}
\label{sec:perturbation}
For $\delta = (\delta_1, \delta_2, \cdots, \delta_m) \in\R^m$ and $L_A(\delta)$ defined in \eqref{eq:L-def},
define the associated \emph{relative perturbation}
\begin{equation}
\label{eq:H-def}
H_A(\delta)=L_A(\delta)(L_A(0))^{-1}.
\end{equation}
Thus $H_A(\delta)$ isolates the effect of a small perturbation of the driving randomness.
Write
\begin{equation}
\label{eq:Gt-Pt}
P_{t+1}=R(i_m,\theta_m)R(i_{m-1},\theta_{m-1})\cdots R(i_{t+1},\theta_{t+1}),
\qquad
P_{m+1}=\Id,
\end{equation}
and define the transported basis vectors
\begin{equation}
\label{eq:vt-def}
v_t=\Ad(P_{t+1})a_{i_t}\in \so(n).
\end{equation}

\begin{lemma}
We have
\begin{equation}
\label{eq:H-product}
H_A(\delta)=\prod_{t=m}^1 \exp\!\bigl(\sqrt{2}\,\delta_t v_t\bigr).
\end{equation}
\end{lemma}
\begin{proof}
Write \(B_t=R(i_t,\theta_t)\) and
\(E_t=\exp(\sqrt2\,\delta_t a_{i_t})\). We have
\(R(i_t,\theta_t+\delta_t)=E_tB_t\) and thus
\[
H_A(\delta)
=
E_mB_mE_{m-1}B_{m-1}\cdots E_1B_1
B_1^{-1}\cdots B_m^{-1}.
\]
Regrouping the factors gives
\[
H_A(\delta)
=
\prod_{t=m}^{1}
P_{t+1}E_tP_{t+1}^{-1},
\qquad
P_{t+1}=B_m\cdots B_{t+1}.
\]
Since conjugation commutes with the exponential map,
\[
P_{t+1}E_tP_{t+1}^{-1}
=
\exp\!\bigl(\sqrt2\,\delta_t\,\Ad(P_{t+1})a_{i_t}\bigr)
=
\exp\!\bigl(\sqrt2\,\delta_t v_t\bigr),
\]
which proves \eqref{eq:H-product}.
\end{proof}

Define the (principal) logarithm $u_A(\delta)$ of $H_A(\delta)$ and its derivative:
\begin{equation} \label{eq:u-def}
u_A(\delta)=\log H_A(\delta),
\qquad
J_A=\dd u_A(0):\R^m\to \so(n).
\end{equation}

\begin{lemma}
\label{lem:jacobian}
The derivative $J_A$ satisfies
\begin{equation}
\label{eq:J-columns}
J_A e_t = \sqrt{2}\,v_t=\sqrt{2}\,\Ad(P_{t+1})a_{i_t},
\end{equation}
and therefore
\begin{equation}
\label{eq:J-linear}
J_A\delta = \sqrt{2}\sum_{t=1}^m \delta_t v_t.
\end{equation}
\end{lemma}

\begin{proof}
Fix $t$ and set $\delta=se_t$.
All factors in \eqref{eq:H-product} are equal to the identity except the $t^\mathrm{th}$ factor, so
\[
H_A(se_t)=\exp\!\bigl(\sqrt{2}\,s v_t\bigr).
\]
For $s$ near $0$, the principal logarithm of the right-hand side is $\sqrt{2}\,s v_t$.
Differentiating at $s=0$ gives \eqref{eq:J-columns}, and summing over coordinates gives
\eqref{eq:J-linear}.
\end{proof}

Next, we use Lemma \ref{lem:jacobian} to obtain good control on $u_A(\delta)$ for sufficiently small perturbations: 

\begin{proposition}\label{prop:quadratic}
There exist absolute constants $c_0,C_0>0$ such that if
\begin{equation}
\label{eq:radius}
\norm{\delta}_2\le \frac{c_0}{\sqrt{m}},
\end{equation}
the principal logarithm $u_A(\delta)=\log H_A(\delta)$ is well defined. Furthermore, 
whenever \eqref{eq:radius} holds, $u_A(\delta)$ admits the decomposition:
\begin{equation}
\label{eq:quadratic-expansion}
u_A(\delta):=J_A\delta+r_A(\delta)
\end{equation}
with $\norm{r_A(\delta)}_{\HS}\le C_0 m\norm{\delta}_2^2$ and 
\begin{equation}
\label{eq:quadratic-derivative}
\norm{\dd u_A(\delta)-J_A}_{\op}\le C_0 m\norm{\delta}_2.
\end{equation}
\end{proposition}

\begin{proof}
For $1\le t\le m$, set
\[
E_t(s)=\exp\!\bigl(\sqrt{2}\,s v_t\bigr),
\]
so that
\[
H_A(\delta)=E_m(\delta_m)\cdots E_1(\delta_1).
\]
We first verify that the principal logarithm of $H_A(\delta)$ is well-defined. Since $\norm{v_t}_{\HS}=1$ and $\norm{v_t}_{\op}\le 1$,
\[
\sum_{t=1}^m \norm{\sqrt{2}\,\delta_t v_t}_{\op}
\le
\sqrt{2}\sum_{t=1}^m \abs{\delta_t}
\le
\sqrt{2m}\,\norm{\delta}_2.
\]
If \eqref{eq:radius} holds and $c_0$ is small enough, then
$\sum_t \norm{\sqrt{2}\,\delta_t v_t}_{\op}\le 1/8$.
Therefore
\begin{equation}
\label{eq:H-close-to-I}
\norm{H_A(\delta)-\Id}_{\op}
\le
\exp\!\Bigl(\sum_{t=1}^m \norm{\sqrt{2}\,\delta_t v_t}_{\op}\Bigr)-1
\le \frac14.
\end{equation}
Hence the principal logarithm is well defined; see \cite[Chapter~1]{high:FM}.

We next bound the first and second derivatives of $H_A$.
Differentiating \eqref{eq:H-product} gives
\[
\partial_j H_A(\delta)
=
E_m(\delta_m)\cdots E_{j+1}(\delta_{j+1})
\,E_j'(\delta_j)\,
E_{j-1}(\delta_{j-1})\cdots E_1(\delta_1).
\]
Applying \eqref{eq:radius} and the fact that $\norm{E_j'(s)}_{\op}\le \sqrt{2}\,e^{\sqrt{2}\abs{s}}$, there exists a universal constant $C$ such that for all $j$
\[
\norm{\partial_j H_A(\delta)}_{\op}\le C.
\]
Consequently, for all $\eta\in\R^m$,
\[
\dd H_A(\delta)[\eta]
=
\sum_{j=1}^m \eta_j\,\partial_j H_A(\delta),
\]
so
\[
\norm{\dd H_A(\delta)[\eta]}_{\op}
\le
C\sum_{j=1}^m \abs{\eta_j}
\le C\sqrt{m}\,\norm{\eta}_2.
\]
Therefore
\begin{equation}
\label{eq:dH-bound}
\norm{\dd H_A(\delta)}_{\op}
\le C\sqrt{m}.
\end{equation}

Differentiating \eqref{eq:H-product} once more, if $k\neq j$ then $\partial_k\partial_j H_A(\delta)$ is a product of the factors $E_\ell(\delta_\ell)$ in which the $j^\mathrm{th}$ factor is replaced by $E_j'(\delta_j)$ and the $k^\mathrm{th}$ factor is replaced by $E_k'(\delta_k)$, while if $k=j$ then the $j$th factor is replaced by $E_j''(\delta_j)$. Since $\norm{E_\ell'(s)}_{\op}\le C$ and $\norm{E_\ell''(s)}_{\op}\le C'$ on the ball defined by \eqref{eq:radius}, there exists a universal constant $C''$ such that, for all $j,k$,
\[
\norm{\partial_k\partial_j H_A(\delta)}_{\op}\le C''.
\]
Therefore, for all $\eta,\zeta\in\R^m$,
\[
\norm{\dd^2 H_A(\delta)[\eta,\zeta]}_{\op}
\le
C''\sum_{j,k=1}^m \abs{\eta_j}\abs{\zeta_k}
\le
C''m\,\norm{\eta}_2\norm{\zeta}_2.
\]
Hence
\begin{equation}
\label{eq:d2H-bound}
\norm{\dd^2 H_A(\delta)}_{\op}
\le
C''m.
\end{equation}
Let
\[
\mathcal U=\{M\in M_n(\R): \norm{M-\Id}_{\op}\le 1/2\}.
\]
The matrix logarithm is smooth on $\mathcal U$, so its first and second derivatives are bounded in $\mathcal U$
by an absolute constant.
Combining this with \eqref{eq:H-close-to-I}, \eqref{eq:dH-bound}, and \eqref{eq:d2H-bound},
the chain rule yields
\begin{equation}
\label{eq:d2u-bound}
\sup_{\norm{\delta}_2\le c_0/\sqrt{m}}\norm{\dd^2 u_A(\delta)}_{\op}\le C_0 m
\end{equation}
for some absolute constant $C_0$.

Recall that $u_A(0)=0$, and that by Lemma~\ref{lem:jacobian}  $\dd u_A(0) = J_A$.
Taylor's theorem gives
\[
u_A(\delta)=J_A\delta+\int_0^1 (1-s)\,\dd^2u_A(s\delta)[\delta,\delta]\,\dd s,
\]
which together with \eqref{eq:d2u-bound} proves \eqref{eq:quadratic-expansion}.
Bound \eqref{eq:d2u-bound} and the mean value theorem give
\[
\norm{\dd u_A(\delta)-J_A}_{\op}
\le
\sup_{\norm{\xi}_2\le c_0/\sqrt m}\norm{\dd^2u_A(\xi)}_{\op}\,\norm{\delta}_2
\le C_0 m\norm{\delta}_2,
\]
which is \eqref{eq:quadratic-derivative} and the proof is finished.
\end{proof}

\section{Covariance and the random-matrix input}
\label{sec:covariance}

In Proposition \ref{prop:quadratic} of Section~\ref{sec:perturbation}, we showed that the logarithmic map
\[
u_A(\delta)=\log H_A(\delta)
\]
admits the linear approximation 
\[
u_A(\delta) \approx J_A\delta, \quad J_A=\dd u_A(0)
\] for small deterministic $\delta$.  Define the matrix:
\begin{equation}\label{eqn:MA}
M_A=J_AJ_A^\top=2\sum_{t=1}^m v_t v_t^\top,
\qquad v_t=\Ad(P_{t+1})a_{i_t}
\end{equation}
where $P_t$ is as defined in \eqref{eq:Gt-Pt}.
We show that $M_A$ is well-conditioned with high probability by exploiting a conditional isotropy property: conditional on the updates after time $t$ in the block, the randomness of $i_t$ implies that $v_t v_t^\top$ has conditional expectation $(1/N)I$. This yields a matrix martingale difference decomposition, to which we apply Tropp's matrix Freedman inequality \cite[Theorem~1.2]{Tropp} to control the spectrum of $M_A$.

\begin{proposition}
\label{prop:covariance}
The event
\begin{equation} \label{eq:good-event}
\mathcal G_A
=
\left\{
\frac{m}{N}\le \lambda_{\min}(M_A)\le \lambda_{\max}(M_A)\le 3\frac{m}{N}
\right\}
\end{equation}
satisfies
\[
\Prob(\mathcal G_A^c)=o(1),
\]
where the probability is over the randomness of the update block
$
A=((i_t,\theta_t))_{t=1}^m.
$
\end{proposition}

\begin{proof}
We identify $\so(n)$ with $\R^N$ through the ordered orthonormal basis
$(a_i)_{i=1}^N$. Recall from \eqref{eq:vt-def} that
\[
v_t=\Ad(P_{t+1})a_{i_t}\in \so(n),
\]
and define
\begin{equation}
\label{eq:C-def-explicit}
C_A=\sum_{t=1}^m v_tv_t^\top.
\end{equation}
By Lemma~\ref{lem:jacobian}, we have
\[
J_A e_t=\sqrt{2}\,v_t,
\]
and therefore
\begin{equation}
\label{eq:M-from-C-explicit}
M_A=J_AJ_A^\top=2\sum_{t=1}^m v_tv_t^\top=2C_A.
\end{equation}

We now introduce the backwards filtration
\[
\mathcal F_t=\sigma\bigl(\{(i_s,\theta_s)\}_{s\ge t}\bigr),
\qquad 1\le t\le m+1.
\]
Since $P_{t+1}$ is $\mathcal F_{t+1}$-measurable and $i_t$ is uniform on
$\{1,\dots,N\}$ independently of $\mathcal F_{t+1}$, we obtain
\begin{align}
\E\bigl[v_tv_t^\top\mid \mathcal F_{t+1}\bigr]
&=
\Ad(P_{t+1})
\Bigl(
\frac1N\sum_{i=1}^N a_ia_i^\top
\Bigr)
\Ad(P_{t+1})^\top \notag\\
&=
\frac1N I_N.
\label{EqMartDef}
\end{align}
Define
\begin{equation}
\label{eq:Yt-def-explicit}
Y_t=v_tv_t^\top-\frac1N I_N.
\end{equation}
Then
\[
\E(Y_t\mid \mathcal F_{t+1})=0,
\]
so, after reversing time, $(Y_t)$ is a self-adjoint matrix martingale
difference sequence.

Since $v_tv_t^\top$ is a rank-one orthogonal projection, its eigenvalues are
$1,0,\dots,0$. 
Consequently,
\begin{equation}
\label{eq:Yt-norm-explicit}
\|Y_t\|_{\op}=1-\frac1N\le 1.
\end{equation}

Next, using $(v_tv_t^\top)^2=v_tv_t^\top$, we compute
\begin{align*}
Y_t^2
&=
\Bigl(v_tv_t^\top-\frac1N I_N\Bigr)^2 \\
&=
\Bigl(1-\frac2N\Bigr)v_tv_t^\top+\frac1{N^2}I_N.
\end{align*}
Taking conditional expectation and using
\eqref{EqMartDef} gives
\begin{align}
\E(Y_t^2\mid \mathcal F_{t+1})
&=
\Bigl(1-\frac2N\Bigr)\frac1N I_N+\frac1{N^2}I_N \notag\\
&=
\Bigl(\frac1N-\frac1{N^2}\Bigr)I_N
\preceq
\frac1N I_N.
\label{eq:Yt-var-explicit}
\end{align}
Hence the predictable quadratic variation satisfies
\begin{equation}
\label{eq:Wm-explicit}
W_m:=\sum_{t=1}^m \E(Y_t^2\mid \mathcal F_{t+1})
\preceq \frac{m}{N}I_N.
\end{equation}

We now apply Tropp's matrix Freedman inequality \cite[Theorem~1.2]{Tropp}.
Using \eqref{eq:Yt-norm-explicit} and \eqref{eq:Wm-explicit}, for every
$s>0$,
\[
\Prob\!\left(
\lambda_{\max}\Bigl(\sum_{t=1}^m Y_t\Bigr)\ge s
\right)
\le
N\exp\!\left(
-\frac{s^2/2}{m/N+s/3}
\right).
\]
Applying the same bound to $-Y_t$ yields
\begin{equation}
\label{eq:freedman-both-sides-explicit}
\Prob\!\left(
\left\|\sum_{t=1}^m Y_t\right\|_{\op}\ge s
\right)
\le
2N\exp\!\left(
-\frac{s^2/2}{m/N+s/3}
\right).
\end{equation}
From \eqref{eq:C-def-explicit} and \eqref{eq:Yt-def-explicit},
\[
C_A
=
\sum_{t=1}^m v_tv_t^\top
=
\frac{m}{N}I_N+\sum_{t=1}^m Y_t.
\]
Choose $s=\frac12\frac{m}{N}$ so that
\begin{equation}
\label{eq:CA-concentration-explicit}
\Prob\!\left(
\left\|C_A-\frac{m}{N}I_N\right\|_{\op}\ge \frac12\frac{m}{N}
\right)
\le
2N\exp\!\left(-\frac{3}{28}\frac{m}{N}\right).
\end{equation}
Define the event
\[
\mathcal E
:=
\left\{
\left\|C_A-\frac{m}{N}I_N\right\|_{\op}
<
\frac12\frac{m}{N}
\right\}.
\]
Then by \eqref{eq:CA-concentration-explicit},
\[
\Prob(\mathcal E^c)
\le
2N\exp\!\left(-\frac{3}{28}\frac{m}{N}\right).
\]
On the event $\mathcal E$, we have
\[
\frac12\frac{m}{N}I_N
\preceq
C_A
\preceq
\frac32\frac{m}{N}I_N.
\]
Using \eqref{eq:M-from-C-explicit}, this implies
\[
\frac{m}{N}I_N
\preceq
M_A
\preceq
3\frac{m}{N}I_N.
\]
Equivalently,
\[
\frac{m}{N}\le \lambda_{\min}(M_A)\le \lambda_{\max}(M_A)\le 3\frac{m}{N},
\]
so that $\mathcal E \subseteq \mathcal G_A$. Therefore,
\[
\Prob(\mathcal G_A^c)
\le
\Prob(\mathcal E^c)
\le
2N\exp\!\left(-\frac{3}{28}\frac{m}{N}\right).
\]
Since $m \ge C_1 N \log N$ and $C_1 > 10 > \frac{28}{3}$ , we obtain
\[
\Prob(\mathcal G_A^c)
\le
2N^{\,1-\frac{3}{28}C_1} = o(1)
\]
and the proof is finished.
\end{proof}
\begin{remark} \label{rem:keymallrm}
The key input in Proposition~\ref{prop:covariance} is the conditional isotropy identity \eqref{EqMartDef},
\[
\E\bigl[v_t v_t^\top \mid \mathcal F_{t+1}\bigr] = \frac{1}{N} I,
\]
which follows from two structural features of Kac’s walk: the uniform choice of the coordinate direction $i_t$, and the fact that the adjoint action $\Ad(g)$ is orthogonal on $\so(n)$. This construction suggests a general approach for proving mixing bounds for high-dimensional local-update Markov chains: if the update directions are sufficiently randomized and their transport does not distort directions too severely, one can analyze the spread of the chain via the covariance of transported directions and control it using matrix concentration.
\end{remark}

\section{Gaussian approximation of the perturbation block}\label{sec:transport}

We now show that the perturbation block is quite close to an appropriate Gaussian. Recall from Section~\ref{subsec:global-notation} that we fix a fresh block of updates
\[
A=((i_1,\theta_1),\dots,(i_m,\theta_m)),
\]
and introduce an independent Gaussian perturbation
\[
\Delta \sim \mathcal N(0,\sigma_n^2 I_m).
\]
The endpoint of the perturbed block is described by
\[
H_A(\Delta)=L_A(\Delta)L_A(0)^{-1},
\qquad
U_A := u_A(\Delta)=\log H_A(\Delta)\in \so(n).
\]

By Lemma~\ref{lem:marginal}, $H_A(\Delta)$ has the correct marginal law for the $m$-step update. By Lemma \ref{lem:local-log}, the logarithm is a bijection on a reasonably large ball, so (after a small truncation) it turns out to suffice to analyze the distribution of $U_A$. Throughout this section, we condition on the block $A$ and study the law of $U_A$ with respect to the Gaussian randomness in $\Delta$.

From Section~\ref{sec:perturbation}, we have the expansion
\[
u_A(\delta)=J_A\delta + r_A(\delta),
\]
valid for $\delta$ in a small neighborhood of the origin, while Section~\ref{sec:covariance} shows that the covariance
\[
M_A = J_A J_A^\top
\]
is well-conditioned with high probability over $A$. Together, these results suggest that $U_A$ should be well approximated by the Gaussian random variable
\[
J_A \Delta \sim \mathcal N(0,\sigma_n^2 M_A).
\]
To make this precise, we restrict to the high-probability event
\begin{equation} \label{eq:E-def}
E = \Big \{\|\Delta\|_2 \le 2\sigma_n \sqrt{m} \Big \}.
\end{equation}
Under our choice of parameters from \eqref{eq:scales}, we have
\[
2\sigma_n\sqrt{m} = 2N^{-3}\sqrt{m} \ll m^{-1/2},
\]
so the radius condition \eqref{eq:radius} holds on $E$ for all sufficiently large $n$. In particular, $U_A$ is well-defined on $E$, and the linear approximation above is valid on $E$.
Set
\begin{equation}
\label{eq:Sigma-gamma-def}
\Sigma_A := \sigma_n^2 M_A,
\qquad
\gamma_A := \mathcal N(0,\Sigma_A).
\end{equation}
We will compare the conditional law $\Law(U_A \mid A; E)$ to the Gaussian measure
$\gamma_A$ and show that they are close in total variation; see Proposition~\ref{prop:gaussian-approx}. We commit a small abuse of notation when describing the conditional distribution  $\Law(U_A \mid A; E)$: by conditioning on the random variable $A$ we are conditioning on the $\sigma$-algebra generated by $A$, while by conditioning on $E$ we are conditioning on the event $E$ happening (rather than conditioning on the $\sigma$-algebra generated by the $E$). This causes no real difficulties as $\sigma(E), \sigma(A)$ are independent.

We begin with a Gaussian stability estimate for small perturbations of the
identity.

\begin{proposition}[Gaussian stability under small perturbations]
\label{prop:near-identity}
Let $\gamma_d=\mathcal N(0,I_d)$. Let $R:\R^d\to\R^d$ be a $C^1$ map satisfying
\[
\sup_{x\in\R^d}\|R(x)\|_2\le \alpha,
\qquad
\sup_{x\in\R^d}\|\dd R(x)\|_{\op}\le \beta
\]
with $0<\beta\le 1/2$. Set $T=\Id+R$. Then $T$ is a global $C^1$
diffeomorphism, and there exists a universal constant $C_{\mathrm{NI}}>0$ such
that
\begin{equation}
\label{eq:near-identity-kl}
\KL(T_\#\gamma_d\,\|\,\gamma_d)
\le
C_{\mathrm{NI}}\bigl(\alpha^2+d\beta^2\bigr),
\end{equation}
and
\begin{equation}
\label{eq:near-identity-tv}
\|T_\#\gamma_d-\gamma_d\|_{\TV}
\le
C_{\mathrm{NI}}\bigl(\alpha+\sqrt d\,\beta\bigr).
\end{equation}
\end{proposition}

\begin{proof}
First we show that $T$ is a global diffeomorphism. For $x,y\in\R^d$,
\[
\|T(x)-T(y)\|_2
\ge
\|x-y\|_2-\|R(x)-R(y)\|_2
\ge
(1-\beta)\|x-y\|_2.
\]
Thus $T$ is injective. For surjectivity, fix $z\in\R^d$ and consider
\[
\Phi_z(x)=z-R(x).
\]
Since $\Phi_z$ is a contraction with Lipschitz constant at most $\beta<1$, it has
a unique fixed point. This fixed point satisfies $T(x)=z$. Hence $T$ is bijective.
Moreover,
\[
\dd T(x)=I+\dd R(x)
\]
is invertible for every $x$, so the inverse function theorem implies that $T$ is
a global $C^1$ diffeomorphism.

Let $X\sim\gamma_d$. By the change-of-variables formula,
\[
\frac{\dd(T_\#\gamma_d)}{\dd\gamma_d}(T(x))
=
\frac{\exp(-\|x\|_2^2/2)}
{\exp(-\|T(x)\|_2^2/2)\det \dd T(x)}.
\]
Therefore
\begin{equation}
\label{eq:entropy-identity}
\KL(T_\#\gamma_d\,\|\,\gamma_d)
=
\E\left[
\frac{\|T(X)\|_2^2-\|X\|_2^2}{2}
-
\log\det \dd T(X)
\right].
\end{equation}

Set $B(x)=\dd R(x)$. Since $\|B(x)\|_{\op}\le 1/2$, the series
\[
\log\det(I+B)
=
\Tr\log(I+B)
=
\Tr B+\sum_{k\ge2}\frac{(-1)^{k+1}}{k}\Tr(B^k)
\]
converges absolutely. We now claim that there is a universal constant $C_{\log}$ such that
\begin{equation}
\label{eq:logdet-near-identity}
-\log\det(I+B)+\Tr B
\le
C_{\log}\|B\|_{\HS}^2 .
\end{equation}
To see this, note that
\[
|\Tr(B^k)|
\le
\|B\|_{\HS}^2\|B\|_{\op}^{k-2},
\qquad k\ge 2,
\]
and the geometric series is bounded uniformly when $\|B\|_{\op}\le 1/2$.
Using $T(X)=X+R(X)$ in \eqref{eq:entropy-identity}, so
\[
\KL(T_\#\gamma_d\,\|\,\gamma_d)
\le
\E\bigl[\langle X,R(X)\rangle-\Tr(\dd R(X))\bigr]
+
\frac12\E\|R(X)\|_2^2
+
C_{\log}\E\|\dd R(X)\|_{\HS}^2.
\]
The first term vanishes by Gaussian integration by parts:
\[
\E\langle X,R(X)\rangle
=
\E\operatorname{div}R(X)
=
\E\Tr(\dd R(X)).
\]
Hence
\[
\KL(T_\#\gamma_d\,\|\,\gamma_d)
\le
\frac12\alpha^2+C_{\log}d\beta^2.
\]
Choose
\[
C_{\mathrm{NI}}\ge \max\{1/2,C_{\log}\}.
\]
This proves \eqref{eq:near-identity-kl}. 

Pinsker's inequality gives
\[
\|T_\#\gamma_d-\gamma_d\|_{\TV}
\le
\sqrt{\frac12 C_{\mathrm{NI}}\bigl(\alpha^2+d\beta^2\bigr)}
\le
C_{\mathrm{NI}}\bigl(\alpha+\sqrt d\,\beta\bigr),
\]
after increasing \(C_{\mathrm{NI}}\) once and for all. This proves
\eqref{eq:near-identity-tv} and we are done.
\end{proof}
We will use the following functions to smoothly truncate our distributions:

\begin{lemma}[Radial cutoff]
\label{lem:radial-cutoff}
There is an absolute constant $C_{\mathrm{cut}}<\infty$ such that for every
$d\ge 1$ and every $R>0$, there exists a smooth function
$\Psi_{d,R}:\R^d\to[0,1]$ satisfying
\[
\Psi_{d,R}(x)=1 \quad \text{if } \|x\|_2\le 2R,
\qquad
\Psi_{d,R}(x)=0 \quad \text{if } \|x\|_2\ge 3R,
\]
and
\[
\sup_{x\in\R^d}\|\nabla\Psi_{d,R}(x)\|_2
\le
\frac{C_{\mathrm{cut}}}{R}.
\]
\end{lemma}

\begin{proof}
Choose a smooth function $\eta:[0,\infty)\to[0,1]$ such that
\[
\eta(s)=1 \quad \text{for } s\le 4,
\qquad
\eta(s)=0 \quad \text{for } s\ge 9.
\]
Set
\[
\Psi_{d,R}(x)=\eta\!\left(\frac{\|x\|_2^2}{R^2}\right).
\]
Then
\[
\nabla\Psi_{d,R}(x)
=
\eta'\!\left(\frac{\|x\|_2^2}{R^2}\right)\frac{2x}{R^2}.
\]
The derivative can be nonzero only when $2R\le \|x\|_2\le 3R$, so
\[
\|\nabla\Psi_{d,R}(x)\|_2
\le
\frac{6\|\eta'\|_\infty}{R}.
\]
Thus the claim holds with $C_{\mathrm{cut}}=6\|\eta'\|_\infty$.
\end{proof}

We now use Proposition \ref{prop:near-identity} to analyze our second stage coupling. Recall the Gaussian measure $\gamma_{A}$ from Equation \eqref{eq:Sigma-gamma-def} and the event $\mathcal G_A$ defined in \eqref{eq:good-event}.
\begin{proposition}[Gaussian approximation for the perturbation block]
\label{prop:gaussian-approx}
There exist absolute constants $c_{\mathrm{GA}},C_{\mathrm{GA}}>0$ such that, for all sufficiently large $n$, the following holds uniformly over all $A\in\mathcal G_A$:
\[
\|\Law(U_A\mid A;E)-\gamma_A\|_{\TV}
\le
C_{\mathrm{GA}}\frac{\sigma_n m^2}{\sqrt{m/N}}+e^{-c_{\mathrm{GA}}m}.
\]
In particular, under the scale choice \eqref{eq:scales}, the right-hand side is $o(1)$.
\end{proposition}
\begin{proof}
Let $Z = \sigma_{n}^{-1} \Delta$ so that $Z\sim\mathcal N(0,I_m)$. From \eqref{eq:E-def}, recall the event
\[
E=\{\|\Delta\|_2\le 2\sigma_n\sqrt m\}
=\{\|Z\|_2\le 2\sqrt m\}.
\]
By Gaussian concentration, there exists an absolute constant
$c_{\mathrm{tail}}>0$ such that
\begin{equation}
\label{eq:chi-tail-clean}
\Prob(E^c)\le e^{-c_{\mathrm{tail}}m}.
\end{equation}

We introduce a cutoff so that the logarithmic expansion from
Proposition~\ref{prop:quadratic} is used only inside its domain. Using the construction in
Lemma~\ref{lem:radial-cutoff}, choose
\[
\chi=\Psi_{m,\sigma_n\sqrt m}.
\]
Then
\[
\chi(\delta)=1 \quad \text{if } \|\delta\|_2\le 2\sigma_n\sqrt m,
\qquad
\chi(\delta)=0 \quad \text{if } \|\delta\|_2\ge 3\sigma_n\sqrt m,
\]
and
\begin{equation}
\label{eq:cutoff-grad-bound}
\sup_{\delta\in\R^m}\|\nabla\chi(\delta)\|_2
\le
\frac{C_{\mathrm{cut}}}{\sigma_n\sqrt m}.
\end{equation}
Define
\begin{equation}
\label{eq:trunc-defs-clean}
\widetilde r_A(\delta):=\chi(\delta)r_A(\delta),
\qquad
\widetilde u_A(\delta):=J_A\delta+\widetilde r_A(\delta),
\qquad
\widetilde U_A:=\widetilde u_A(\Delta).
\end{equation}
Since $\chi\equiv 1$ on $E$,
\begin{equation}
\label{eq:UA-tildeUA-on-E}
\widetilde U_A=U_A
\qquad\text{on }E.
\end{equation}

We next record the bounds on the truncated remainder. By
Proposition~\ref{prop:quadratic}, there is an absolute constant
$C_{0}>0$ such that, whenever
\[
\|\delta\|_2\le c_0/\sqrt m,
\]
we have
\[
\|r_A(\delta)\|_{\HS}\le C_{0}m\|\delta\|_2^2,
\qquad
\|\dd r_A(\delta)\|_{\op}\le C_{0}m\|\delta\|_2.
\]
Under \eqref{eq:scales},
\[
3\sigma_n\sqrt m\le c_0/\sqrt m
\]
for all sufficiently large $n$. Hence the preceding bounds apply on
$\mathrm{supp}(\chi)$. We now claim that there exist absolute constants
$C_r,C_{dr}>0$ such that
\begin{equation}
\label{eq:trunc-r-clean}
\sup_{\delta\in\R^m}\|\widetilde r_A(\delta)\|_{\HS}
\le
C_r\sigma_n^2m^2,
\end{equation}
and
\begin{equation}
\label{eq:trunc-dr-clean}
\sup_{\delta\in\R^m}\|\dd\widetilde r_A(\delta)\|_{\op}
\le
C_{dr}\sigma_n m^{3/2}.
\end{equation}
To see this, on $\mathrm{supp}(\chi)$,
\[
\|r_A(\delta)\|_{\HS}\lesssim \sigma_n^2m^2,
\qquad
\|\dd r_A(\delta)\|_{\op}\lesssim \sigma_n m^{3/2}.
\]
Moreover,
\[
\dd\widetilde r_A
=
\chi\,\dd r_A+(\nabla\chi)\otimes r_A.
\]
The term $\chi\,\dd r_A$ is controlled by the bound on $\dd r_A$, while $(\nabla\chi)\otimes r_A$ is controlled by \eqref{eq:cutoff-grad-bound} together with the bound on $r_A$

Fix now $A\in\mathcal G_A$. By the definition of the good set $\mathcal G_A$,
\[
\lambda_{\min}(M_A)\ge \frac{m}{N}.
\]
Let
\[
\rho_A:=s_{\min}(J_A)=\sqrt{\lambda_{\min}(M_A)}.
\]
Thus
\begin{equation}
\label{eq:rho-lower-clean}
\rho_A\ge \sqrt{m/N}.
\end{equation}
Take a singular value decomposition
\[
J_A=Q[\Lambda\ \ 0]W^\top,
\]
where $Q\in O(N)$, $W\in O(m)$, and
\[
\Lambda=\diag(s_1,\dots,s_N),
\qquad
s_i\ge \rho_A.
\]
Since $W^\top Z\sim\mathcal N(0,I_m)$, write
\[
W^\top Z=(X,Y),
\qquad
X\in\R^N,\quad Y\in\R^{m-N},
\]
where $X\sim\mathcal N(0,I_N)$ and $Y\sim\mathcal N(0,I_{m-N})$ are independent.
Since $\Delta=\sigma_n Z=\sigma_n W(X,Y)$,
\begin{equation}
\label{eq:linear-gaussian-law-clean}
J_A\Delta=\sigma_n Q\Lambda X,
\qquad
\Law(J_A\Delta)=\mathcal N(0,\Sigma_A)=\gamma_A.
\end{equation}

For fixed $y\in\R^{m-N}$, define
\[
T_{A,y}(x)
:=
x+(\sigma_n Q\Lambda)^{-1}
\widetilde r_A\bigl(\sigma_n W(x,y)\bigr),
\qquad x\in\R^N,
\]
where we identify $\so(n)$ with $\R^N$ via the fixed orthonormal basis $(a_i)_{i=1}^N$.
Then
\begin{equation}
\label{eq:trunc-factorization-clean}
\widetilde U_A
=
\sigma_n Q\Lambda\,T_{A,Y}(X).
\end{equation}
We now verify that $T_{A,y}$ is a small perturbation of the identity, uniformly
in $y$. By \eqref{eq:trunc-r-clean} and \eqref{eq:rho-lower-clean},
\[
\sup_{x\in\R^N}\|T_{A,y}(x)-x\|_2
\le
\frac{C_r\sigma_n^2m^2}{\sigma_n\rho_A}
\le
C_\alpha\frac{\sigma_n m^2}{\sqrt{m/N}},
\]
where $C_\alpha$ is an absolute constant. Similarly, by
\eqref{eq:trunc-dr-clean} and \eqref{eq:rho-lower-clean},
\[
\sup_{x\in\R^N}\|\dd T_{A,y}(x)-I\|_{\op}
\le
\frac{C_{dr}\sigma_n m^{3/2}}{\rho_A}
\le
C_\beta\frac{\sigma_n m^{3/2}}{\sqrt{m/N}},
\]
where $C_\beta$ is an absolute constant.

Set
\[
\alpha_A:=C_\alpha\frac{\sigma_n m^2}{\sqrt{m/N}},
\qquad
\beta_A:=C_\beta\frac{\sigma_n m^{3/2}}{\sqrt{m/N}}.
\]
Since
\[
\beta_A=C_\beta\sigma_n m\sqrt N\to0,
\]
we have $\beta_A\le 1/2$ for all sufficiently large $n$.

Applying Proposition~\ref{prop:near-identity} in dimension $N$, with
$\gamma_N=\mathcal N(0,I_N)$, gives, for every $y\in\R^{m-N}$,
\[
\|(T_{A,y})_\#\gamma_N-\gamma_N\|_{\TV}
\le
C_{\mathrm{NI}}\bigl(\alpha_A+\sqrt N\,\beta_A\bigr).
\]
Since $m\ge N$,
\[
\sqrt N\,\beta_A
=
C_\beta\frac{\sigma_n m^{3/2}\sqrt N}{\sqrt{m/N}}
\le
C_\beta\frac{\sigma_n m^2}{\sqrt{m/N}}.
\]
Therefore, for an absolute constant $C_{\mathrm{tr}}$,
\begin{equation}
\label{eq:trunc-gaussian-clean}
\|(T_{A,y})_\#\gamma_N-\gamma_N\|_{\TV}
\le
C_{\mathrm{tr}}\frac{\sigma_n m^2}{\sqrt{m/N}}
\qquad\text{for every }y\in\R^{m-N}.
\end{equation}

By \eqref{eq:linear-gaussian-law-clean}, \eqref{eq:trunc-factorization-clean},
and invariance of total variation under the invertible linear map $x\mapsto \sigma_n Q\Lambda x$,
\[
\|\Law(\widetilde U_A\mid A,Y=y)-\gamma_A\|_{\TV}
\le
C_{\mathrm{tr}}\frac{\sigma_n m^2}{\sqrt{m/N}}.
\]
Averaging over $Y$ gives
\begin{equation}
\label{eq:tildeUA-to-gammaA-clean}
\|\Law(\widetilde U_A\mid A)-\gamma_A\|_{\TV}
\le
C_{\mathrm{tr}}\frac{\sigma_n m^2}{\sqrt{m/N}}.
\end{equation}

It remains to remove the cutoff. Since $\widetilde U_A=U_A$ on $E$,
there exists a probability measure $\eta_A$ such that
\[
\Law(\widetilde U_A\mid A)
=
\Prob(E)\Law(U_A\mid A;E)+\Prob(E^c)\eta_A.
\]
Hence
\[
\|\Law(U_A\mid A;E)-\Law(\widetilde U_A\mid A)\|_{\TV}
\le
2\Prob(E^c)
\le
2e^{-c_{\mathrm{tail}}m}.
\]
Combining this with \eqref{eq:tildeUA-to-gammaA-clean} gives
\[
\|\Law(U_A\mid A;E)-\gamma_A\|_{\TV}
\le
C_{\mathrm{tr}}\frac{\sigma_n m^2}{\sqrt{m/N}}
+
2e^{-c_{\mathrm{tail}}m}.
\]
Now set
\[
C_{\mathrm{GA}}:=C_{\mathrm{tr}},
\qquad
c_{\mathrm{GA}}:=\frac12 c_{\mathrm{tail}}.
\]
Since \(m\to\infty\), for all sufficiently large \(n\),
\[
2e^{-c_{\mathrm{tail}}m}\le e^{-c_{\mathrm{GA}}m}.
\]
Therefore
\[
\|\Law(U_A\mid A;E)-\gamma_A\|_{\TV}
\le
C_{\mathrm{GA}}\frac{\sigma_n m^2}{\sqrt{m/N}}
+
e^{-c_{\mathrm{GA}}m}.
\]

Finally, under \eqref{eq:scales},
$
\frac{\sigma_n m^2}{\sqrt{m/N}}=o(1),
$
so the right-hand side above is $o(1)$ and this proves the proposition.
\end{proof}

\section{The transport step on the group}
\label{sec:group-transport}

In the previous section, we proved a Gaussian approximation for $U_A$; see Proposition~\ref{prop:gaussian-approx}. In this section, we transfer that logarithmic approximation to a total-variation bound for the group-valued law of $H_A(\Delta)$; see Proposition~\ref{prop:block-translate}.

\noindent
The following bound is well-known:
\begin{lemma}
\label{lem:gaussian-shift}
If $G\sim \mathcal N(0,\Sigma)$ on $\R^N$ with $\Sigma$ positive definite, then for every
$h\in\R^N$,
\begin{equation}
\label{eq:gaussian-shift}
\norm{\Law(G+h)-\Law(G)}_{\TV}
\le
\frac{1}{\sqrt{2}}\norm{\Sigma^{-1/2}h}_2.
\end{equation}
\end{lemma}

\begin{proof}
The Kullback--Leibler divergence between two Gaussians with the same covariance is
\[
\KL\!\bigl(\mathcal N(h,\Sigma)\,\|\,\mathcal N(0,\Sigma)\bigr)
=
\frac12 h^\top \Sigma^{-1}h
=
\frac12 \norm{\Sigma^{-1/2}h}_2^2.
\]
Pinsker's inequality then gives \eqref{eq:gaussian-shift}.
\end{proof}

We next estimate the nonlinear correction created by passing from addition in the Lie algebra $\so(n)$
to right multiplication in the group $\SO(n)$.

\begin{lemma}
\label{lem:bch-remainder}
There exist constants $r_\ast,C_\ast>0$ independent of $n$ such that the map
\[
\Phi(\xi,h) \equiv \log\!\bigl(\exp(\xi)\exp(h)\bigr)
\]
is smooth on
\[
\{(\xi,h)\in \so(n)\times \so(n): \norm{\xi}_{\HS}\le r_\ast,\ \norm{h}_{\HS}\le r_\ast\},
\]
and the remainder
\begin{equation} \label{EqDefRemainderLemma6}
\beta_h(\xi)=\Phi(\xi,h)-\xi-h
\end{equation}
satisfies
\begin{equation}
\label{eq:bch-remainder}
\norm{\beta_h(\xi)}_{\HS}\le C_\ast \norm{\xi}_{\HS}\norm{h}_{\HS},
\qquad
\norm{\dd_\xi \beta_h(\xi)}_{\op}\le C_\ast \norm{h}_{\HS}
\end{equation}
whenever $\norm{\xi}_{\HS},\norm{h}_{\HS}\le r_\ast$.
\end{lemma}

\begin{proof}
Choose $r_\ast>0$ small enough that
\[
e^{2r_\ast}-1\le \frac12.
\]
If $\norm{\xi}_{\HS},\norm{h}_{\HS}\le r_\ast$, then $\norm{\xi}_{\op},\norm{h}_{\op}\le r_\ast$, and therefore
\begin{align*}
\norm{\exp(\xi)\exp(h)-\Id}_{\op}
&\le
\norm{\exp(\xi)}_{\op}\,\norm{\exp(h)-\Id}_{\op}+\norm{\exp(\xi)-\Id}_{\op} \\
&\le e^{\norm{\xi}_{\op}}\bigl(e^{\norm{h}_{\op}}-1\bigr)+\bigl(e^{\norm{\xi}_{\op}}-1\bigr) \\
& \le e^{2r_\ast}-1 \le \frac12.
\end{align*}
Hence $\exp(\xi)\exp(h)$ lies in the fixed operator-norm neighborhood
\[
\mathcal U=\{M\in M_n(\R): \norm{M-\Id}_{\op}\le 1/2\},
\]
on which the principal logarithm is analytic; see \cite[Chapter~1]{high:FM}. The matrix exponential, matrix multiplication, and principal logarithm all have first and second Fr\'echet derivatives bounded by absolute constants on these fixed neighborhoods (with respect to the Hilbert--Schmidt norm). Consequently $\Phi$ is smooth on the stated product ball, and its mixed second derivatives are bounded there by an absolute constant.

Define $\beta(\xi,h)=\Phi(\xi,h)-\xi-h$.
Because $\Phi(\xi,0)=\xi$ and $\Phi(0,h)=h$, we have
\[
\beta(\xi,0)=0,
\qquad
\beta(0,h)=0.
\]
Therefore
\[
\beta(\xi,h)
=
\int_0^1\!\int_0^1 \partial_1\partial_2\beta(s\xi,th)[\xi,h]\,\dd s\,\dd t,
\]
so the fact that the mixed second-derivatives are bounded gives, for some $C_{\ast}$, the inequality
\[
\norm{\beta_h(\xi)}_{\HS}\le C_\ast \norm{\xi}_{\HS}\norm{h}_{\HS}.
\]
Also $\dd_\xi\beta(\xi,0)=0$ for all $\xi$, hence
\[
\dd_\xi \beta_h(\xi)
=
\int_0^1 \partial_2\dd_\xi\beta(\xi,th)[h]\,\dd t,
\]
and the same bound yields
\[
\norm{\dd_\xi \beta_h(\xi)}_{\op}\le C_\ast \norm{h}_{\HS}
\]
and the proof is finished.
\end{proof}
For $g\in \SO(n)$, write
\[
T_g(x)=xg.
\]

The next proposition relates translations in the tangent space $\so(n)$ to translation in the group $\SO(n)$: if a law on $\SO(n)$ is the exponential pushforward of a measure on
$\so(n)$ that is close to a centered Gaussian, then small right translation on
the group has small total-variation cost.

\begin{proposition}[Gaussian overlap transfer on $\SO(n)$]
\label{prop:gaussian-bch}
Let $\mu$ be a probability measure on $\so(n)$ supported on
\begin{equation} \label{SupportHyp1}
\{\xi\in \so(n): \norm{\xi}_{\HS}\le r_\ast/4\},
\end{equation}
 where $r_\ast$ is the radius from Lemma~\ref{lem:bch-remainder},
and let $\nu=\exp_\#\mu.$
Let $\gamma=\mathcal N(0,\Sigma)$
be a centered Gaussian law on $\so(n)$, where $\Sigma$ is positive definite, and
assume
\begin{equation}
\label{eq:gaussian-transfer-chart}
\norm{\Sigma^{1/2}}_{\op}\le \frac{r_\ast}{16\sqrt N}.
\end{equation}
Set $\kappa(\Sigma)=\frac{\lambda_{\max}(\Sigma)}{\lambda_{\min}(\Sigma)}$.
There exist absolute constants $c_{\mathrm{GT}},c_B,C_B>0$ with the following
property: for every $h\in \so(n)$ satisfying
\begin{equation}
\label{eq:gaussian-transfer-h-small}
\norm{h}_{\HS}\le c_{\mathrm{GT}}\,\kappa(\Sigma)^{-1/2},
\end{equation}
and with $g=\exp(h)$, 
\begin{align}
\norm{\nu-(T_g)_\#\nu}_{\TV}
\le
&2\norm{\mu-\gamma}_{\TV}
+
C_B\norm{\Sigma^{-1/2}h}_{\HS} \notag \\
&+
C_B\sqrt N\,\kappa(\Sigma)^{1/2}\norm{h}_{\HS}
+
e^{-c_BN}.\label{eq:gaussian-bch}
\end{align}
\end{proposition}
\begin{proof}
Let \(r_\ast\) and \(C_\ast\) be the constants from
Lemma~\ref{lem:bch-remainder}. Define
\[
c_{\mathrm{GT}}
:=
\min\left\{
r_\ast,\,
\frac{1}{2C_\ast(1+3C_{\mathrm{cut}})}
\right\}.
\]
Since \(\kappa(\Sigma)\ge 1\), the assumption
\[
\|h\|_{\HS}\le c_{\mathrm{GT}}\kappa(\Sigma)^{-1/2}
\]
implies
\[
\|h\|_{\HS}\le c_{\mathrm{GT}}\le r_\ast.
\]

Choose a measurable extension
\[
B_h:\so(n)\to\so(n)
\]
of the local map
\[
\xi\mapsto \log(\exp(\xi)\exp(h))
\]
from the ball \(\{\|\xi\|_{\HS}\le r_\ast/4\}\) to all of \(\so(n)\).
Since \(\mu\) is supported in this ball, we have, for \(\xi\in \mathrm{supp}\,\mu\),
\[
B_h(\xi)=\log(\exp(\xi)\exp(h)).
\]
Thus
\[
T_g(\exp(\xi))
=
\exp(\xi)\exp(h)
=
\exp(B_h(\xi)),
\]
and hence
\[
\nu=\exp_\#\mu,
\qquad
(T_g)_\#\nu=\exp_\#(B_h)_\#\mu.
\]
Since total variation cannot increase under pushforward,
\[
\|\nu-(T_g)_\#\nu\|_{\TV}
\le
\|\mu-(B_h)_\#\mu\|_{\TV}.
\]
Inserting \(\gamma\) and \((B_h)_\#\gamma\), and then using contraction of
total variation under the pushforward by the measurable map \(B_h\), gives
\begin{align}
\|\mu-(B_h)_\#\mu\|_{\TV}
&\le
\|\mu-\gamma\|_{\TV}
+
\|\gamma-(B_h)_\#\gamma\|_{\TV}
+
\|(B_h)_\#\gamma-(B_h)_\#\mu\|_{\TV}
\notag\\
&\le
\|\mu-\gamma\|_{\TV}
+
\|\gamma-(B_h)_\#\gamma\|_{\TV}
+
\|\gamma-\mu\|_{\TV}
\notag\\
&=
2\|\mu-\gamma\|_{\TV}
+
\|\gamma-(B_h)_\#\gamma\|_{\TV}.
\label{eq:transfer-triangle}
\end{align}

It remains to control the second term. We split it into a mean shift and a nonlinear Baker-Campbell-Hausdorff correction; the latter is localized and whitened so that Proposition~\ref{prop:near-identity} can be applied globally. Let
\[
\tau_h(\xi)=\xi+h.
\]
We have
\begin{equation}
\label{eq:transfer-split}
\|\gamma-(B_h)_\#\gamma\|_{\TV}
\le
\|\gamma-(\tau_h)_\#\gamma\|_{\TV}
+
\|(\tau_h)_\#\gamma-(B_h)_\#\gamma\|_{\TV}.
\end{equation}
By Lemma~\ref{lem:gaussian-shift},
\begin{equation}
\label{eq:transfer-gaussian-shift}
\|\gamma-(\tau_h)_\#\gamma\|_{\TV}
\le
\frac1{\sqrt2}\|\Sigma^{-1/2}h\|_{\HS}.
\end{equation}
This controls the first term in \eqref{eq:transfer-split}. We now bound the nonlinear term. Set
\[
\kappa:=\kappa(\Sigma)
=
\frac{\lambda_{\max}(\Sigma)}{\lambda_{\min}(\Sigma)}.
\]
Write
\[
\Xi=\Sigma^{1/2}X,
\qquad
X\sim\mathcal N(0,I_N),
\]
so that \(\Xi\sim\gamma\). Let
\[
F=\{\|X\|_2\le 2\sqrt N\}.
\]
A standard Chernoff bound gives
\begin{equation}
\label{eq:transfer-X-tail}
\Prob(F^c)\le e^{-2c_BN},
\end{equation}
where
\[
c_B:=\frac{3-\log 4}{4}.
\]
Using the construction and notation from the statement of Lemma~\ref{lem:radial-cutoff}, choose
\[
\Psi=\Psi_{N,\sqrt N}.
\]
Then
\[
\Psi(x)=1 \quad\text{if }\|x\|_2\le 2\sqrt N,
\qquad
\Psi(x)=0 \quad\text{if }\|x\|_2\ge 3\sqrt N,
\]
and
\begin{equation}
\label{eq:transfer-cutoff-grad}
\sup_{x\in\R^N}\|\nabla\Psi(x)\|_2
\le
C_{\mathrm{cut}}N^{-1/2}.
\end{equation}
Recall that $r_\ast$ is the radius from Lemma~\ref{lem:bch-remainder}. We first check that the whitened variable stays in this logarithmic chart on the ball that contains the support of the cutoff.
By the chart assumption \eqref{eq:gaussian-transfer-chart}, if
\(\|x\|_2\le 4\sqrt N\), then
\[
\|\Sigma^{1/2}x\|_{\HS}
\le
\|\Sigma^{1/2}\|_{\op}\|x\|_2
\le
\frac{r_\ast}{4}.
\]
Thus, on this ball, the map
\[
x\mapsto \log\!\bigl(\exp(\Sigma^{1/2}x)\exp(h)\bigr)
\]
is well-defined and smooth.

We next write the Baker-Campbell-Hausdorff remainder in whitened coordinates and then use the cutoff function to obtain the global bounds required by Proposition~\ref{prop:near-identity}.
Define the local whitened error
\[
b_h(x):=\Sigma^{-1/2}\beta_h(\Sigma^{1/2}x),
\qquad \|x\|_2\le 4\sqrt N,
\]
where
\[
\beta_h(\xi)=\log(\exp(\xi)\exp(h))-\xi-h
\]
is the local remainder from Lemma~\ref{lem:bch-remainder}. 
By the chain rule,
\[
\dd b_h(x)
=
\Sigma^{-1/2}\circ
\dd_\xi\beta_h(\Sigma^{1/2}x)
\circ \Sigma^{1/2}.
\]
Since
\[
\|\Sigma^{-1/2}\|_{\op}\|\Sigma^{1/2}\|_{\op}
=
\sqrt{\kappa},
\]
Lemma~\ref{lem:bch-remainder} gives, for \(\|x\|_2\le 4\sqrt N\),
\begin{equation}
\label{eq:transfer-b-bound}
\|b_h(x)\|_2
\le
C_\ast\sqrt{\kappa}\,\|x\|_2\|h\|_{\HS},
\end{equation}
and
\begin{equation}
\label{eq:transfer-db-bound}
\|\dd b_h(x)\|_{\op}
\le
C_\ast\sqrt{\kappa}\,\|h\|_{\HS}.
\end{equation}

Define the global cutoff error
\[
R_h(x):=
\begin{cases}
\Psi(x)b_h(x), & \|x\|_2\le 4\sqrt N,\\
0, & \|x\|_2>4\sqrt N,
\end{cases}
\qquad
S_h(x):=x+R_h(x).
\]
$S_{h}$ is well-defined and \(C^1\), because \(\Psi(x)=0\) when
\(\|x\|_2\ge 3\sqrt N\). Since \(0\le \Psi\le 1\) and \(\Psi\) is supported in
\(\{\|x\|_2\le 3\sqrt N\}\), \eqref{eq:transfer-b-bound} gives
\begin{equation}
\label{eq:transfer-R-bound}
\sup_{x\in\R^N}\|R_h(x)\|_2
\le
3C_\ast\sqrt{\kappa N}\,\|h\|_{\HS}.
\end{equation}

We next bound the derivative. On the support of \(\Psi\),
\[
\dd R_h(x)
=
\Psi(x)\dd b_h(x)+(\nabla\Psi(x))\otimes b_h(x),
\]
and \(\dd R_h(x)=0\) outside this support. The first term is controlled by
\eqref{eq:transfer-db-bound}:
\[
\|\Psi(x)\dd b_h(x)\|_{\op}
\le
C_\ast\sqrt{\kappa}\,\|h\|_{\HS}.
\]
For the second term, on the support of \(\nabla\Psi\) we have
\(\|x\|_2\le 3\sqrt N\). Therefore \eqref{eq:transfer-cutoff-grad} and
\eqref{eq:transfer-b-bound} imply
\[
\|(\nabla\Psi(x))\otimes b_h(x)\|_{\op}
\le
C_{\mathrm{cut}}N^{-1/2}
\cdot
3C_\ast\sqrt{\kappa N}\,\|h\|_{\HS}.
\]
Combining the two estimates gives
\begin{equation}
\label{eq:transfer-DR-bound}
\sup_{x\in\R^N}\|\dd R_h(x)\|_{\op}
\le
C_\ast(1+3C_{\mathrm{cut}})\sqrt{\kappa}\,\|h\|_{\HS}.
\end{equation}

By the definition of \(c_{\mathrm{GT}}\), the assumption on \(h\) implies
\[
\sup_{x\in\R^N}\|\dd R_h(x)\|_{\op}\le \frac12.
\]
Therefore Proposition~\ref{prop:near-identity} applies to
\[
S_h=\Id+R_h.
\]
With
\[
\alpha=3C_\ast\sqrt{\kappa N}\,\|h\|_{\HS},
\qquad
\beta=C_\ast(1+3C_{\mathrm{cut}})\sqrt{\kappa}\,\|h\|_{\HS},
\]
it yields
\begin{equation}
\label{eq:transfer-Sh-close}
\|(S_h)_\#\mathcal N(0,I_N)-\mathcal N(0,I_N)\|_{\TV}
\le
C_{\mathrm{NI}}C_\ast(4+3C_{\mathrm{cut}})
\sqrt{\kappa N}\,\|h\|_{\HS}.
\end{equation}

We have now controlled the whitened cutoff map \(S_h=\Id+R_h\). It remains to
compare this cutoff model with the original map \(B_h\). Define
\[
\widetilde B_h(x)=\Sigma^{-1/2}B_h(\Sigma^{1/2}x),
\qquad
u=\Sigma^{-1/2}h.
\]
On the event \(F\), the cutoff is inactive: \(\Psi(X)=1\). Moreover,
\(\|\Sigma^{1/2}X\|_{\HS}\le r_\ast/4\), so \(B_h\) agrees with the local
logarithmic map at \(\Sigma^{1/2}X\). Hence, on \(F\),
\[
B_h(\Sigma^{1/2}X)
=
\log\!\bigl(\exp(\Sigma^{1/2}X)\exp(h)\bigr)
=
\Sigma^{1/2}X+h+\beta_h(\Sigma^{1/2}X),
\]
and therefore
\[
\widetilde B_h(X)
=
X+u+b_h(X)
=
\tau_u(S_h(X)),
\qquad
\tau_u(x)=x+u.
\]
Thus the coupling inequality gives
\[
\|\Law(\widetilde B_h(X))-\Law(\tau_u(S_h(X)))\|_{\TV}
\le
\Prob(F^c).
\]
By the triangle inequality and translation invariance of total variation,
\begin{align*}
\|\Law(\widetilde B_h(X))-\Law(\tau_u(X))\|_{\TV}
&\le
\Prob(F^c)
+
\|\Law(\tau_u(S_h(X)))-\Law(\tau_u(X))\|_{\TV}\\
&=
\Prob(F^c)
+
\|(S_h)_\#\mathcal N(0,I_N)-\mathcal N(0,I_N)\|_{\TV}.
\end{align*}
Using \eqref{eq:transfer-X-tail} in the form
\(\Prob(F^c)\le e^{-2c_BN}\), together with
\eqref{eq:transfer-Sh-close}, gives
\begin{align}
\|(\widetilde B_h)_\#\mathcal N(0,I_N)
-(\tau_u)_\#\mathcal N(0,I_N)\|_{\TV}
&\le
C_{\mathrm{NI}}C_\ast(4+3C_{\mathrm{cut}})
\sqrt{\kappa N}\,\|h\|_{\HS}
+
e^{-2c_BN}.
\label{eq:transfer-whitened-nonlinear}
\end{align}
Finally, applying the invertible linear map \(\Sigma^{1/2}\) to both measures and
using
\[
B_h\circ \Sigma^{1/2}=\Sigma^{1/2}\circ \widetilde B_h,
\qquad
\tau_h\circ \Sigma^{1/2}=\Sigma^{1/2}\circ \tau_u,
\]
turns \eqref{eq:transfer-whitened-nonlinear} into
\begin{equation}
\label{eq:transfer-nonlinear}
\|(B_h)_\#\gamma-(\tau_h)_\#\gamma\|_{\TV}
\le
C_{\mathrm{NI}}C_\ast(4+3C_{\mathrm{cut}})
\sqrt{\kappa N}\,\|h\|_{\HS}
+
e^{-2c_BN}.
\end{equation}

Finally set
\[
C_B:=
\max\left\{
\frac1{\sqrt2},
C_{\mathrm{NI}}C_\ast(4+3C_{\mathrm{cut}})
\right\}.
\]
Combining \eqref{eq:transfer-triangle}, \eqref{eq:transfer-split},
\eqref{eq:transfer-gaussian-shift}, and \eqref{eq:transfer-nonlinear}, we get
\[
\|\nu-(T_g)_\#\nu\|_{\TV}
\le
2\|\mu-\gamma\|_{\TV}
+
C_B\|\Sigma^{-1/2}h\|_{\HS}
+
C_B\sqrt{\kappa N}\,\|h\|_{\HS}
+
e^{-c_BN}.
\]
Since \(\kappa=\kappa(\Sigma)\), this is exactly \eqref{eq:gaussian-bch} and the proof is finished.
\end{proof}

Denote the ``block law" by
\begin{equation}
\label{eq:nuA-def}
\nu_A=\Law(H_A(\Delta)\mid A).
\end{equation}
We now apply Proposition \ref{prop:gaussian-bch} to compare $\nu_A$ and $(T_g)_\#\nu_A$.

\begin{proposition}
\label{prop:block-translate}
Let $c_{\log}$ be the constant from Lemma \ref{lem:local-log}. Let $h\in\so(n)$ and let $g=\exp(h)$.  If $\norm{h}_{\HS}\le c_{\log}\omega_n$, then
for all sufficiently large $n$ and for all $A\in \mathcal G_A$,
\begin{equation} \label{eqn:tvo(1)bd}
\norm{\nu_A-(T_g)_\#\nu_A}_{\TV}=o(1).
\end{equation}
\end{proposition}
\begin{proof}
Fix a realization $A\in \mathcal G_A$. All probabilities and conditional laws
below are taken with this $A$ fixed.

Recall that
\[
E=\{\|\Delta\|_2\le 2\sigma_n\sqrt m\}
\]
is the typical-size event for the Gaussian perturbation, and write
\[
\mu:=\mu_A^E=\Law(U_A\mid A;E),
\qquad
\nu:=\nu_A^E=\Law(H_A(\Delta)\mid A;E).
\]
Also recall that
\[
\gamma:=\gamma_A=\mathcal N(0,\Sigma_A),
\qquad
\Sigma_A:=\sigma_n^2M_A.
\]
Since $E$ depends only on $\Delta$, which is independent of $A$, we may write
\[
\nu_A=\Prob(E)\nu+\Prob(E^c)\nu_A^{E^c}
\]
for some probability measure $\nu_A^{E^c}$. Hence
\begin{equation}
\label{eq:block-decomposition}
\|\nu_A-(T_g)_\#\nu_A\|_{\TV}
\le
\|\nu-(T_g)_\#\nu\|_{\TV}
+
2\Prob(E^c).
\end{equation}
We now wish to apply Proposition~\ref{prop:gaussian-bch} with
\[
\mu=\mu_A^E,
\qquad
\nu=\nu_A^E,
\qquad
\gamma=\gamma_A,
\qquad
\Sigma=\Sigma_A.
\]
To do so, we verify the three assumptions of Proposition~\ref{prop:gaussian-bch}: $\mu$ is supported in the logarithmic chart, the chart-size bound on $\Sigma_A$, and the smallness condition on $h$.

\smallskip
\noindent
\emph{(i) Support in the logarithmic chart.}
Let $r_\ast$ be the chart radius from Proposition~\ref{prop:gaussian-bch}, and let
$C_0$ be the constant from Proposition~\ref{prop:quadratic}. Proposition~\ref{prop:gaussian-bch} requires that $\mu$ be supported in a set of the form \eqref{SupportHyp1}. We now check that this ``support" hypothesis holds.

On $E$,
\[
\|\Delta\|_2\le 2\sigma_n\sqrt m.
\]
Under the scale choice \eqref{eq:scales},
\[
2\sigma_n\sqrt m=o(m^{-1/2}),
\]
so for all sufficiently large $n$ the radius condition in
Proposition~\ref{prop:quadratic} holds on $E$. Therefore
\[
U_A=J_A\Delta+r_A(\Delta).
\]
On $\mathcal G_A$,
\[
\|J_A\|_{\op}
=
\sqrt{\lambda_{\max}(M_A)}
\le
\sqrt{3m/N}.
\]
Hence, on $E$,
\[
\|U_A\|_{\HS}
\le
\|J_A\|_{\op}\|\Delta\|_2+\|r_A(\Delta)\|_{\HS}
\le
2\sqrt3\,\sigma_n\frac{m}{\sqrt N}
+
4C_0\sigma_n^2m^2.
\]
The right-hand side is $o(1)$ under \eqref{eq:scales}. Since $r_\ast>0$ is fixed,
it follows that for all sufficiently large $n$,
\[
\|U_A\|_{\HS}\le \frac{r_\ast}{4}
\qquad\text{on }E.
\]
Thus $\mu$ is supported in
\[
\{\xi\in\so(n):\|\xi\|_{\HS}\le r_\ast/4\}.
\]
Moreover, on $E$ we have $H_A(\Delta)=\exp(U_A)$, and therefore
\[
\nu=\exp_\#\mu.
\]
This verifies the ``support" hypothesis of Proposition~\ref{prop:gaussian-bch}.

\smallskip
\noindent
\emph{(ii) Chart-size condition for the Gaussian covariance.}
On $\mathcal G_A$,
\[
\frac{m}{N}I_N\preceq M_A\preceq 3\frac{m}{N}I_N,
\]
so
\[
\sigma_n^2\frac{m}{N}I_N
\preceq
\Sigma_A
\preceq
3\sigma_n^2\frac{m}{N}I_N.
\]
In particular,
\[
\|\Sigma_A^{1/2}\|_{\op}
\le
\sqrt3\,\sigma_n\sqrt{m/N}.
\]
Therefore
\[
16\sqrt N\,\|\Sigma_A^{1/2}\|_{\op}
\le
16\sqrt3\,\sigma_n\sqrt m
=o(1),
\]
so for all sufficiently large $n$,
\[
\|\Sigma_A^{1/2}\|_{\op}\le \frac{r_\ast}{16\sqrt N}.
\]
Thus the chart-size hypothesis \eqref{eq:gaussian-transfer-chart} of
Proposition~\ref{prop:gaussian-bch} holds.

\smallskip
\noindent
\emph{(iii) Smallness condition on $h$.}
Let $c_{\mathrm{GT}},C_B,c_B$ be the constants from
Proposition~\ref{prop:gaussian-bch}. On $\mathcal G_A$,
\[
\kappa(\Sigma_A)
=
\frac{\lambda_{\max}(\Sigma_A)}{\lambda_{\min}(\Sigma_A)}
\le 3.
\]
Since $\kappa(\Sigma_A)\le 3$,
the smallness hypothesis \eqref{eq:gaussian-transfer-h-small} from
Proposition~\ref{prop:gaussian-bch},
\[
\|h\|_{\HS}\le c_{\mathrm{GT}}\kappa(\Sigma_A)^{-1/2},
\]
is implied by
\[
\|h\|_{\HS}\le \frac{c_{\mathrm{GT}}}{\sqrt3}.
\]
But by assumption,
\[
\|h\|_{\HS}\le c_{\log}\omega_n=c_{\log}N^{-4}\to 0,
\]
so the bound $|h|_{\HS}\le c_{\mathrm{GT}}/\sqrt3$ holds for all sufficiently large $n$. Therefore the smallness hypothesis of Proposition~\ref{prop:gaussian-bch} is satisfied.

Having verified the hypotheses, we may apply
Proposition~\ref{prop:gaussian-bch} with
\[
\mu=\mu_A^E,\qquad
\nu=\nu_A^E,\qquad
\gamma=\gamma_A,\qquad
\Sigma=\Sigma_A.
\]
This yields
\begin{align*}
\|\nu_A^E-(T_g)_\#\nu_A^E\|_{\TV}
\le
2\|\mu_A^E-\gamma_A\|_{\TV}
&+
C_B\|\Sigma_A^{-1/2}h\|_{\HS} \\
&+
C_B\sqrt N\,\kappa(\Sigma_A)^{1/2}\|h\|_{\HS}
+
e^{-c_BN}.
\end{align*}
Using again that on \(\mathcal G_A\),
\[
\kappa(\Sigma_A)\le 3,
\qquad
\|\Sigma_A^{-1/2}h\|_{\HS}
\le
\frac{\|h\|_{\HS}}{\sigma_n\sqrt{m/N}},
\]
we obtain
\begin{equation}
\label{eq:block-transfer}
\|\nu-(T_g)_\#\nu\|_{\TV}
\le
2\|\mu-\gamma\|_{\TV}
+
C_B\frac{\|h\|_{\HS}}{\sigma_n\sqrt{m/N}}
+
\sqrt3\,C_B\sqrt N\,\|h\|_{\HS}
+
e^{-c_BN}.
\end{equation}
Next, Proposition~\ref{prop:gaussian-approx} gives
\[
\|\mu-\gamma\|_{\TV}
\le
C_{\mathrm{GA}}\frac{\sigma_n m^2}{\sqrt{m/N}}
+
e^{-c_{\mathrm{GA}}m},
\]
where \(c_{\mathrm{GA}},C_{\mathrm{GA}}>0\) are the constants from that proposition.
Substituting this into \eqref{eq:block-transfer}, we get
\begin{align*}
\|\nu-(T_g)_\#\nu\|_{\TV}
\le\;&
2C_{\mathrm{GA}}\frac{\sigma_n m^2}{\sqrt{m/N}}
+
C_B\frac{\|h\|_{\HS}}{\sigma_n\sqrt{m/N}}\\
&+
\sqrt3\,C_B\sqrt N\,\|h\|_{\HS}
+
2e^{-c_{\mathrm{GA}}m}
+
e^{-c_BN}.
\end{align*}

Finally, since \(\Delta=\sigma_n Z\) with \(Z\sim\mathcal N(0,I_m)\), a standard
Gaussian tail bound gives an absolute constant \(c_E>0\) such that
\[
\Prob(E^c)\le e^{-c_Em}.
\]
Combining this with \eqref{eq:block-decomposition}, we conclude that
\begin{align}
\label{eq:finest}
\|\nu_A-(T_g)_\#\nu_A\|_{\TV}
\le\;&
2C_{\mathrm{GA}}\frac{\sigma_n m^2}{\sqrt{m/N}}
+
C_B\frac{\|h\|_{\HS}}{\sigma_n\sqrt{m/N}} \notag\\
&+
\sqrt3\,C_B\sqrt N\,\|h\|_{\HS}
+
2e^{-c_{\mathrm{GA}}m}
+
e^{-c_BN}
+
2e^{-c_Em}.
\end{align}

It remains to prove the estimate \eqref{eqn:tvo(1)bd}. Assume
\[
\|h\|_{\HS}\le c_{\log}\omega_n=c_{\log}N^{-4}.
\]
Under \eqref{eq:scales},
\[
\frac{\sigma_n m^2}{\sqrt{m/N}}
\asymp
\frac{\log^{3/2}N}{N}
=o(1),
\]
\[
\frac{\|h\|_{\HS}}{\sigma_n\sqrt{m/N}}
\le
\frac{c_{\log}N^{-4}}{N^{-3}\sqrt{m/N}}
=
\frac{c_{\log}}{N\sqrt{m/N}}
=o(1),
\]
and
\[
\sqrt N\,\|h\|_{\HS}
\le
c_{\log}N^{-7/2}
=o(1).
\]
The exponential terms are also $o(1)$. Therefore, by \eqref{eq:finest},
\[
\|\nu_A-(T_g)_\#\nu_A\|_{\TV}=o(1),
\]
which proves \eqref{eqn:tvo(1)bd} and the proof is finished.
\end{proof}

\section{Proof of the main theorem}
\label{sec:proof-main}

\begin{proof}[Proof of Theorem~\ref{thm:main}]
Fix arbitrary $x,y\in \SO(n)$. For the first stage, apply Proposition~\ref{prop:scaffold} with $B=4$.
Thus there exists
\[
t_0=O(N\log n)
\]
and a coupling
\[
(\widetilde X_t,\widetilde Y_t)_{0\le t\le t_0}
\]
of two copies of Kac's walk with
\[
\widetilde X_0=x,
\qquad
\widetilde Y_0=y,
\]
such that, with
\[
\mathcal E_0:=\{D_{\HS}(\widetilde X_{t_0},\widetilde Y_{t_0})\le \omega_n\},
\qquad
\omega_n=N^{-4},
\]
we have
\begin{equation}
\label{eq:first-stage-event}
\Prob(\mathcal E_0^c)=o(1).
\end{equation}

We now define the second stage. Sample an independent fresh block
\[
A=((i_1,\theta_1),\dots,(i_m,\theta_m))
\]
as in Definition~\ref{defn:randupdate}, and condition on
\[
(\widetilde X_{t_0},\widetilde Y_{t_0},A).
\]
Set
\[
g:=\widetilde X_{t_0}^{-1}\widetilde Y_{t_0},
\qquad
z:=L_A(0)\widetilde X_{t_0}.
\]
Since \(\widetilde Y_{t_0}=\widetilde X_{t_0}g\) and
\(L_A(\delta)=H_A(\delta)L_A(0)\), we have
\[
L_A(\delta)\widetilde X_{t_0}=H_A(\delta)z,
\qquad
L_A(\delta)\widetilde Y_{t_0}=H_A(\delta)zg.
\]

Next, as in Definition~\ref{defn:randupdate2}, let \(\nu_A\) be the block law from
\eqref{eq:nuA-def}:
\[
\nu_A=\Law(H_A(\Delta)\mid A),
\qquad
\Delta\sim\mathcal N(0,\sigma_n^2I_m),
\]
with \(\Delta\) independent of everything else. By the definition of \(\nu_A\) and
\(L_A(\delta)=H_A(\delta)L_A(0)\), conditionally on
\((\widetilde X_{t_0},\widetilde Y_{t_0},A)\),
\[
\Law(L_A(\Delta)\widetilde X_{t_0})=(T_z)_\#\nu_A,
\qquad
\Law(L_A(\Delta)\widetilde Y_{t_0})=(T_{zg})_\#\nu_A.
\]

Now define
\[
(X_{t_0+m},Y_{t_0+m})
\]
conditionally on \((\widetilde X_{t_0},\widetilde Y_{t_0},A)\) to be a
\emph{maximal coupling} of these two conditional laws. Then
\[
\Law(X_{t_0+m}\mid \widetilde X_{t_0},\widetilde Y_{t_0},A)
=
(T_z)_\#\nu_A,
\]
\[
\Law(Y_{t_0+m}\mid \widetilde X_{t_0},\widetilde Y_{t_0},A)
=
(T_{zg})_\#\nu_A,
\]
and
\begin{equation}
\label{eq:conditional-mismatch}
\Prob\!\bigl(X_{t_0+m}\neq Y_{t_0+m}\mid \widetilde X_{t_0},\widetilde Y_{t_0},A\bigr)
=
\|(T_z)_\#\nu_A-(T_{zg})_\#\nu_A\|_{\TV}.
\end{equation}
Averaging over the independent block \(A\), Lemma~\ref{lem:marginal} gives
\[
\Law(X_{t_0+m}\mid \widetilde X_{t_0},\widetilde Y_{t_0})
=
K^m(\widetilde X_{t_0},\cdot),
\qquad
\Law(Y_{t_0+m}\mid \widetilde X_{t_0},\widetilde Y_{t_0})
=
K^m(\widetilde Y_{t_0},\cdot).
\]

We next estimate this mismatch probability on the good event
\[
\mathcal E_0\cap \mathcal G_A.
\]
On \(\mathcal E_0\), bi-invariance of \(D_{\HS}\) gives
\[
D_{\HS}(\Id,g)
=
D_{\HS}(\widetilde X_{t_0},\widetilde Y_{t_0})
\le
\omega_n.
\]
For all sufficiently large \(n\), Lemma~\ref{lem:local-log} applies, so
\[
g=\exp(h)
\qquad\text{with}\qquad
\|h\|_{\HS}\le c_{\log}\omega_n.
\]
Set
\[
\widehat g:=zgz^{-1}.
\]
Since \(g=\exp(h)\), we have
\[
\widehat g=\exp(\widehat h),
\qquad
\widehat h:=\Ad(z)h.
\]
Since the adjoint action is orthogonal on \(\so(n)\),
\[
\|\widehat h\|_{\HS}=\|h\|_{\HS}\le c_{\log}\omega_n.
\]

Now \(T_{z^{-1}}\) is a measurable bijection of \(\SO(n)\), so total variation is
invariant under pushforward by \(T_{z^{-1}}\). Hence
\[
\|(T_z)_\#\nu_A-(T_{zg})_\#\nu_A\|_{\TV}
=
\|\nu_A-(T_{\widehat g})_\#\nu_A\|_{\TV}.
\]
Since \(A\in\mathcal G_A\) on the good event, Proposition~\ref{prop:block-translate},
applied with \(h=\widehat h\), yields
\[
\|\nu_A-(T_{\widehat g})_\#\nu_A\|_{\TV}=o(1)
\]
uniformly on \(\mathcal E_0\cap \mathcal G_A\). 
Therefore, by \eqref{eq:conditional-mismatch},
\[
\Prob\!\bigl(X_{t_0+m}\neq Y_{t_0+m}\mid \widetilde X_{t_0},\widetilde Y_{t_0},A\bigr)
=
o(1)
\qquad\text{on }\mathcal E_0\cap \mathcal G_A.
\]
Taking expectations and splitting according to \(\mathcal E_0\cap\mathcal G_A\),
\begin{align*}
\Prob(X_{t_0+m}\neq Y_{t_0+m})
&\le
\Prob(\mathcal E_0^c)+\Prob(\mathcal G_A^c) \\
&\quad+
\E\!\left[
\1_{\mathcal E_0\cap\mathcal G_A}
\Prob\!\bigl(X_{t_0+m}\neq Y_{t_0+m}\mid
\widetilde X_{t_0},\widetilde Y_{t_0},A\bigr)
\right]
=o(1),
\end{align*}
where we used \eqref{eq:first-stage-event}, Proposition~\ref{prop:covariance},
and the preceding uniform bound.

Since \(X_{t_0+m}\sim K^{t_0+m}(x,\cdot)\) and
\(Y_{t_0+m}\sim K^{t_0+m}(y,\cdot)\), the coupling inequality gives
\[
\|K^{t_0+m}(x,\cdot)-K^{t_0+m}(y,\cdot)\|_{\TV}
\le
\Prob(X_{t_0+m}\neq Y_{t_0+m})
=
o(1),
\]
uniformly in \(x,y\in\SO(n)\).

Integrating the preceding uniform bound over $y\sim\mu$ and using stationarity of Haar measure,
\[
\mu=\int_{\SO(n)}K^{t_0+m}(y,\cdot)\,\mu(\dd y),
\]
we obtain
\[
\sup_{x\in\SO(n)}
\|K^{t_0+m}(x,\cdot)-\mu\|_{\TV}
=
o(1).
\]
Finally,
\[
t_0+m
=
O(N\log n)+O(N\log N)
=
O(n^2\log n),
\]
because \(N=\binom n2\) and \(\log N\asymp \log n\). This proves the theorem.
\end{proof}

\section*{Acknowledgments}

Much of this paper was developed in collaboration with GPT Pro 5.4. In particular, the idea of using matrix martingale inequalities and the results of Tropp came from conversations with GPT. NSP thanks Martin Hairer, Jonathan Mattingly, Gareth Roberts and Andrew Stuart for introducing him to the ideas of Malliavin calculus and the CRiSM postdoctoral fellowship, and Pedja Neskovic and Sriram Subramanian for long standing support. This work was funded by ONR. AS thanks Persi Diaconis for introducing him to this problem and Yunjiang Jiang for many conversations about it. Both authors thank Vinod Vaikuntanathan for helpful discussions.

\bibliographystyle{alpha}
\bibliography{refs}

\end{document}